\documentclass{amsart}

\usepackage[fontsize=11pt]{scrextend}
\usepackage[normalem]{ulem}
\usepackage[svgnames,x11names]{xcolor}
\usepackage{amssymb,amsfonts}
\usepackage[pagebackref,colorlinks,urlcolor={DarkOliveGreen4},
linkcolor={DarkSlateBlue},
citecolor={Chocolate4}]{hyperref}
\usepackage[capitalize]{cleveref}
\usepackage{amssymb,textcomp}
\usepackage{graphicx}
\numberwithin{equation}{section}
\usepackage{bbm}
\usepackage{enumerate}
\usepackage[utf8]{inputenc}
\usepackage[T1]{fontenc}
\usepackage[mathscr]{eucal}
\usepackage{etoolbox}
\usepackage {tikz}
\usetikzlibrary {positioning}
\usepackage[a4paper, twoside=false, vmargin={2cm,3cm}, includehead]{geometry}
\usepackage{nameref}
\newtheorem{lemma}{Lemma}[section]
\newtheorem{theorem}[lemma]{Theorem}
\newtheorem*{theorem*}{Theorem}
\newtheorem{corollary}[lemma]{Corollary}

\newtheorem*{question*}{Open question}
\newtheorem{proposition}[lemma]{Proposition}
\newtheorem*{proposition*}{Proposition}

\newtheorem*{problem*}{Problem}
\theoremstyle{definition}
\newtheorem{definition}[lemma]{Definition}
\newtheorem*{claim*}{Claim}

\newtheorem*{conjecture*}{Conjecture}

\newtheorem{remark}[lemma]{Remark}
\newtheorem*{remark*}{Remark}

\usepackage{framed}
\usepackage{comment}

\makeatletter
\theoremstyle{plain}
\newtheorem*{namedthm}{\namedthmname}
\newcounter{namedthm}
	

\newcommand{\C}{{\mathbb C}}

\newcommand{\N}{{\mathbb N}}

\newcommand{\Q}{{\mathbb Q}}

\newcommand{\T}{{\mathbb T}}
\newcommand{\Z}{{\mathbb Z}}

\title{Uniqueness of a topological Furstenberg system}

\author{Ioannis Kousek}
\email{ioannis.kousek@warwick.ac.uk}
\address{Department of Mathematics, University of Warwick, Coventry, United Kingdom}

\author{Vicente Saavedra-Araya}
\email{vicente.saavedra-araya@warwick.ac.uk}
\address{Department of Mathematics, University of Warwick, Coventry, United Kingdom}

\begin{document}

\begin{abstract}
\small{Given a semigroup $G$ and a bounded function $f: G \to \mathbb{C}$, a topological Furstenberg system of $f$ is a topological dynamical system $\mathbb{X}=(X, (T_g)_{g \in G})$ that encodes the dynamical behaviour of $f$. 
We show that $\mathbb{X}$ is unique up to topological isomorphism, thus providing a topological analogue of the measurable case established by Bergelson and Ferré Moragues for amenable semigroups. We also provide necessary and sufficient conditions for subsets of a
group to have isomorphic Furstenberg systems. In addition, we study sets with minimal Furstenberg systems and identify them 
as a special subclass of dynamically syndetic sets. Moreover, we use this notion to obtain a new characterization of sets of topological recurrence.}
\end{abstract}

\maketitle

\tableofcontents

\section{Introduction}
\label{sec:1}
\subsection{Background}
The use of topological dynamics to obtain results of a Ramsey-theoretic nature
was first implemented by Furstenberg and Weiss in \cite{Furstenberg_Weiss}, where they
provided a dynamical proof of the  multidimensional van der Waerden theorem. This was
achieved using a topological variant of Furstenberg's correspondence
principle—originally introduced by Furstenberg in his seminal ergodic proof
of Szemerédi's theorem \cite{Furstenberg_SzemerediProof}. In this way, they reformulated van der Waerden's theorem as a statement about multiple recurrence in topological dynamics, which they subsequently proved. Since then, variations of this approach have been successfully adapted to
prove a wide range of combinatorial results, including the celebrated
polynomial van der Waerden theorem of Bergelson and Leibman
\cite{Bergelson_Leibman_vdW}, as well as the breakthrough result of Moreira \cite{Moreira} on the partition regularity of $\{x, x+y, xy\}$ in the integers.

The essence of this topological correspondence principle is that, given a 
bounded function $f:\N\to \C$, the statistical 
properties of the sequence $(f(n))_{n\in \N}$ can be 
encoded via 
a topological dynamical system $(X,T)$. In the setting of 
ergodic theory, the measurable version of Furstenberg's 
correspondence, associates to a bounded function $f:\N\to \C$ some measure preserving system 
$(X,\mathcal{B},\mu,T)$,
provided that some (uniform) Ces\'aro averages of the 
sequence $(f(n))_{n\in\N}$ exist. 

These notions can be extended to more general semigroup settings. 
Given a semigroup $G$, a \textit{topological dynamical $G$-system} 
$(X,(T_g)_{g\in G})$ is 
a compact space $X$ together with an action of $G$ on $X$ by 
continuous transformations, meaning that $T_g:X\to X$ is continuous and $T_{gh}=T_gT_h$ for any $g,h\in G$. We sometimes refer to a topological dynamical $G$-system as simply a \textit{system}, and, for brevity, we sometimes refer to the system $(X,(T_g)_{g\in G})$ by $(X,(T_g))$ or by $\mathbb{X}$. Note that, when $G$ is a group, the action is automatically by homeomorphsims as $T_{g^{-1}}=T_{g}^{-1}$, for any $g\in G$, and for the identity element $e$ of $G$, $T_{e}$ is the identity map. 
Throughout this paper, we consider semigroups to which we have added an 
artificial identity element, if necessary (these objects are also known as 
monoids). Moreover, all our semigroups are endowed with the discrete topology. 

A point $a\in X$ is called \textit{transitive} if $\{T_ga: g\in G\}$ is dense in $X$. A system that admits a transitive point is called \textit{transitive} and a system with all its points being transitive is called \textit{minimal}. 

\begin{definition}\label{FS_definition}
Let $G$ be a semigroup and let $f: G \to \C$ be bounded. A dynamical 
$G$-system $(X,(T_g)_{g\in G})$ is called a \emph{topological 
Furstenberg system associated to $f$} if 

\begin{enumerate}[(i)]
    \item \label{cond i} there exist a function $F \in C(X)$ and a transitive point $a \in X$ such that 
    $f(g) = F(T_g a)$ for all $g \in G$, and 
    \item \label{cond iii} for any $x \neq y$ in $X$, there exists $g \in G$ such that 
    $F(T_g x) \neq F(T_g y)$.
\end{enumerate}

By the complex Stone--Weierstrass theorem, condition~(\ref{cond iii}) can be replaced by

\begin{enumerate}[(iii)]
    \item \label{cond ii}   $C(X)$ is the smallest $(T_g)$-invariant $C^{*}$-algebra generated by $F$.\footnote{Here the involution is given by complex conjugation. Equivalently, the smallest $(T_g)$-invariant subalgebra that is closed under complex conjugation and contains $F$ is dense in $C(X)$.} 
\end{enumerate}

If condition~(\ref{cond iii}) is dropped, we will call the system a \emph{generalised topological Furstenberg system}. If condition~(\ref{cond i}) is weakened so that $a\in X$ is not transitive, we speak of a \emph{wasteful Furstenberg system}. 
\end{definition}

For the precise definition of measurable Furstenberg 
systems we refer the reader to \cite[Definition 1.2]{Bergelson_Moragues}.
 When it is clear from the context, we will omit the term 
“topological” and refer to a \textit{Furstenberg system of $f$}. 

For a given semigroup $G$ and a bounded function 
$f : G \to \mathbb{C}$, one can
always construct a symbolic Furstenberg system associated
to $f$. Indeed, this is the original construction 
of Furstenberg and Weiss in \cite{Furstenberg_Weiss} and it motivates the 
more abstract Definition \ref{FS_definition}.

To construct this system, we let $K := \overline{f(G)}$, which is compact by assumption, and consider the space $X:=K^G$ endowed with the product topology. Define an action $R_g : K^G \to K^G$ by
\[
R_g\big( (x(u))_{u \in G} \big) := (x(ug))_{u \in G}, \quad g \in G.
\]
It is straightforward to verify that $R_{gh} = R_g R_h$ 
for all $g,h \in G$, 
so that the \textit{right shift} $(R_g)_{g \in G}$ 
defines an 
action of $G$ on $K^G$. We identify $f$ with the point $(f(g))_{g \in G} \in K^G$, and define
\[
X_{R,f} := \overline{\{ R_h f : h \in G \}} 
= \overline{\{ (f(gh))_{g \in G} : h \in G \}},
\]
where the closure is taken in $K^G$. Setting 
$F(x) = x(e)$, for $x\in X_{R,f}$, condition~(\ref{cond i}) is immediately satisfied, and condition~(\ref{cond iii}) clearly holds. Hence, 
$(X_{R,f}, (R_g))$ 
is a Furstenberg system of $f$.

To determine when two topological systems are essentially equivalent, we use the notion of topological isomorphism (or topological conjugacy). 
Before stating our main results, we recall some standard terminology.
A \textit{factor map} between two topological $G$-systems $(X,(T_g))$ and $(Y,(S_g))$ is a continuous surjection $\phi:X\to Y$ satisfying $\phi \circ T_g=S_g\circ \phi$, meaning that $\phi(T_g(x))=S_g(\phi(x))$, for all $x\in X$ and $g\in G$. In this case, we call $(Y,(S_g))$ a \textit{factor} of $(X,(T_g))$ or the latter an \textit{extension} of the former. 

A \textit{topological isomorphism} between two topological $G$-systems $(X,(T_g))$ and $(Y,(S_g))$ is a homeomorpshim $\phi:X\to Y$ satisfying $\phi \circ T_g=S_g\circ \phi$, for all $g\in G$. Note that, $\phi$ is a factor map from $(X,(T_g))$ to $(Y,(S_g))$, but it also follows that $\phi^{-1} \circ S_g=T_g\circ \phi^{-1}$, and in particular, $\phi^{-1}$ is a factor map from $(Y,(S_g))$ to $(X,(T_g))$ as well. In this case, we say that $(Y,(S_g))$ and $(X,(T_g))$ are \textit{topologically isomorphic}.

We call $(Y,(S_g))$ \textit{a subsystem} of $(X,(T_g))$ if 
$Y$ is a closed invariant subset of $X$, i.e. $T_gY\subset Y$, and $(S_g)=(T_g|_{Y})$, for all $g\in G$. Observe that, a system $(X,(T_g))$ is minimal if and only if it has no non-trivial subsystems (i.e. other than $\mathbb{X}$ itself), as orbit closures are $(T_g)$-invariant and thus give rise to subsystems. 
\subsection{Main results}

\noindent As discussed above, for any semigroup $G$  and any bounded function $f$, one can always construct a corresponding Furstenberg system. It is thus natural to ask whether the symbolic construction yields the \emph{unique} Furstenberg system associated to $f$, in the sense that any other Furstenberg system of $f$ is isomorphic to it. The analogue of this question was addressed by Bergelson and Ferr\'e Moragues \cite{Bergelson_Moragues} in the measurable setting for countable amenable semigroups. They showed that, under natural assumptions, the measurable Furstenberg system associated with a bounded function $f$ is unique up to measurable isomorphism. The lack of the measure space structure in the topological case makes the topological
correspondence more elastic, and we are able to prove the 
following. 

\begin{theorem}\label{FS are isomorphic intro}
Let $G$ be any semigroup and $f:G\to \C$ be bounded. Then, 
any two Furstenberg systems of $f$ are topologically 
isomorphic. 
Moreover, any such system is a factor 
of any generalised Furstenberg 
system of $f$ and there is a maximal generalised 
Furstenberg system of $f$ that contains all the 
other generalised Furstenberg systems as its factors.
\end{theorem}

The construction of the maximal system referred to in the 
previous result is given in Section
\ref{section:Gbeta}, and it naturally involves the Stone-Čech 
compactification of the underlying semigroup. Section 
\ref{sec:2} also includes some basic, necessary facts about 
ultrafilters, other constructions and basic examples of topological 
Furstenberg systems.

Theorem~\ref{FS are isomorphic intro} is proved in Section~\ref{sec:3}, 
where we also establish its analogue for joint Furstenberg systems of 
finite families of bounded functions. At the end of that section, we 
investigate under which conditions on a group $G$ the natural symbolic 
system constructed via left shifts, denoted by 
$(\widetilde{X}_{L,f}, (L_g)_g)$ (see Section~\ref{sec:2.1} for the precise 
construction), is a Furstenberg system of a given bounded function 
$f: G \to \mathbb{C}$. We address this question in general, although the case of infinite groups is more intricate. In particular, even for finite groups the problem already exhibits interesting features, and we prove the following.

\begin{proposition}
Let $G$ be a finite group. Then $(\widetilde{X}_{L,f}, (L_g)_g)$ is a 
Furstenberg system for every function $f: G \to \mathbb{C}$ if and only 
if $G$ is a Dedekind group (i.e., every subgroup is normal).
\end{proposition}

Note that for any semigroup $G$, any set $A \subset G$ can 
be naturally identified with a function $f_A: G \to \{0,1\}$ or a point $a\in \{0,1\}^G$.  
Therefore, we can refer interchangeably to the Furstenberg 
system of $A$, $f_A$ or $a$, and the first question that 
emerges is when the respective Furstenberg systems of 
different subsets are isomorphic. We say 
that two sets $A, B \subset G$ are isomorphic if their 
Furstenberg systems are topologically isomorphic.  

In this direction, we can provide necessary and 
sufficient conditions for $A\subset G$ and $B\subset G$ 
to be isomorphic; we refer the reader to 
Theorem~\ref{characterising isomorphic sets} for the full 
result as it requires some technical 
terminology. As a special case of the main result, we 
have the following sufficient condition, stated in a 
somewhat combinatorial flavour.

\begin{proposition}\label{isomorphic sets intro}
Let $A,B\subset G$. Suppose that for any $k\in \N$,
any $g_1,\ldots,g_k\in G$ and $\epsilon_1,\ldots,\epsilon_k\in \{0,1\}$, we have that
$$(g_1^{-1}A_{\epsilon_1})\cap \cdots (g_k^{-1}A_{\epsilon_k})\neq \emptyset \iff (g_1^{-1}B_{\epsilon_1})\cap \cdots (g_k^{-1}B_{\epsilon_k})\neq \emptyset,$$
where $A_{1}=A$, $A_{0}=G\setminus A$ and $h^{-1}A=\{g\in G: hg\in A\}$. Then $A$ and $B$ are isomorphic.
\end{proposition}

In view of Theorem~\ref{FS are isomorphic intro}, given a bounded function 
$f: G \to \C$, we may speak of \emph{the Furstenberg system of $f$}.  
Of special interest are functions whose Furstenberg systems are minimal. In Section~\ref{sec:minimal}, we focus on subsets of $G$, which correspond to $\{0,1\}$-valued functions, and show that \emph{minimal sets} can be realised as a special subclass of sets of visit times of orbits in minimal systems, known as \emph{dynamically syndetic sets}. This is the context of Theorem~\ref{characterizing minimal sets}.  
Moreover, in Theorem~\ref{minimal colourings}, we provide an analogous characterization for \emph{minimal colourings}, which correspond to finite-valued functions on $G$ whose Furstenberg systems are minimal. As a consequence, we obtain that if a finite colouring of $G$ is minimal, then each colour is a minimal set.

Finally, in Section~\ref{top_recurrence}, we relate the 
minimal sets studied in Section~\ref{sec:minimal} to the 
well-known sets of topological recurrence and provide a 
characterization of these sets using (minimal) 
Furstenberg systems.

\textbf{Acknowledgments.} The authors thank their advisor, Joel Moreira, for 
suggesting this problem and for valuable feedback during 
the developing of this work. We also thank Rigoberto
Zelada for helpful discussions.

\section{Constructions of Furstenberg Systems}
\label{sec:2}
\subsection{Symbolic constructions and other examples}\label{sec:2.1} 

Throughout, unless otherwise specified, $G$ denotes a 
semigroup and $f:G\to \C$ a bounded function. 

We begin with the classical symbolic 
Furstenberg system construction from the introduction, 
adding some details that were glossed over before. We 
let $K := \overline{f(G)}$, which is a compact subset 
of $\mathbb{C}$ by assumption, and endow $K^G$ with 
the product topology. Then, $K^G$ is Hausdorff and, by Tychonoff's theorem, compact.

We consider the full right shift action on $K^G$, that 
is $(K^G, (R_g)_{g \in G})$, where
$$R_g\bigl((x(u))_{u \in G}\bigr) = (x(ug))_{u \in G}.$$
Note that for any $g,h\in G$, we have that
$$R_g(R_h\bigl((x(u))_{u \in G}\bigr) = R_g(x(uh))_{u \in G}=(x(ugh))_{u \in G}=R_{gh}\bigl((x(u))_{u \in G}\bigr).$$
Therefore $(R_g)$ forms a $G$-action on $K^G$.
Recall that the product topology on $K^G$ has a basis 
consisting of cylinder sets, that is, sets of the form
$$
C(U_1,\ldots,U_k; g_1,\ldots,g_k)
= \{ x = (x(g))_{g \in G} \in K^G : x(g_1) \in U_1, \ldots, x(g_k) \in U_k \},
$$
where $k \in \mathbb{N}$, $g_1,\ldots,g_k \in G$, and $U_1,\ldots,U_k \subset K$ are open. When $G$ is a group, it is clear that $R_g$ is a 
bijection on $K^G$, for all $g\in G$, with inverse being $R_{g^{-1}}$. To see that for any semigroup $G$ any of these shifts defines a continuous map, simply note that preimages of cylinder sets under a shift are still cylinder sets. This justifies why $(K^{G},(R_g)_{g\in G})$ is a dynamical $G$-system.

We identify $f$ with a point in $K^G$, still denoted by $f$, and define
\begin{equation}
    X_{R,f}=\overline{\Big\{R_h(f):\ h\in G\Big\}}=\overline{\Big\{(f(gh))_{g\in G}:\ h\in G\Big\}},\label{def:X_R}
\end{equation}
where the closure is taken in the full shift $(K^{G},(R_g)_{g\in G})$. In this system, we consider the continuous function  $F:X_{R,f}\to \mathbb{C}$, given by $F(y)=y_{e}$, where $e$ denotes the identity element of $G$.
Clearly, the $(R_g)$-orbit of $F$ separates points according to condition~(\ref{cond iii}) of Definition \ref{FS_definition}. In addition, the point $f=(f(g))_{g\in G}$ is transitive by construction and $F(R_hf)=f(h)$ for every $h\in G$. Therefore, $(X_{R,f},(R_g)_g)$ is indeed a Furstenberg system of $f$.

When $G$ is a group, another natural symbolic Furstenberg system of $f$ is given by considering 
the left shift action, defined by $L_h(x(g))_{g\in G}=(x(h^{-1}g))_{g\in G}$. In this case, we define $\hat{f}=(f(g^{-1}))_{g\in G}$ and we consider the space
\begin{equation}
    X_{L,f}=\overline{\Big\{L_h(\hat{f}(g))_{g\in G}:\ h\in G\Big\}}=\overline{\Big\{(f(g^{-1}h))_{g\in G}:\ h\in G\Big\}},\label{def:X_L}
\end{equation}
and the orbit closure is taken in the full shift $(K^{G},(L_g)_{g\in G})$.  One can readily check that this defines a Furstenberg system of $f$ via the continuous function $F:X_{L,f}\to \mathbb{C}$ given by $F(y)=y_{e}$. Indeed, $(L_g)$ defines a $G$-action on $X_{L,f}$, the point $\hat{f}$ is transitive for $(X_{L,f},(L_g))$ and for any $h\in G$ we see that 
$$F(L_h\hat{f})=F((\hat{f}(h^{-1}g)_{g\in G}))=\hat{f}(h^{-1})=f(h).$$

What is perhaps less clear is whether the seemingly more natural system with phase space
\begin{equation}
    \Tilde{X}_{L,f}=\overline{\Big\{L_h(f(g))_{g\in G}:\ h\in G\Big\}}=\overline{\Big\{(f(h^{-1}g))_{g\in G}:\ h\in G\Big\}},\label{def:X'_L}
\end{equation}
also defines a Furstenberg system of $f$, via some continuous function $F$ on $\Tilde{X}_{L,f}$ (not necessarily the projection onto the coordinate of the identity) that witnesses the dynamics of $f$ along the orbit of $(f(g))$, or some other $(L_g)$-transitive point in $\Tilde{X}_{L,f}$. It turns out that, for most groups, we can find a function for which this system does not define a Furstenberg system. We explore this question in Section~\ref{sec:left_shifts}.

By considering the symbolic right shift system 
described above, it follows that 
in any semigroup $G$, any subset $A\subset G$ can be 
realised as 
$$A=\{g\in G: T_ga \in U\},$$ 
where $(X,(T_g))$ is a Furstenberg 
system of $A$ (or, equivalently, of $f_A:G\to \{0,1\}$ 
defined via $f(g)=1\iff g\in A$), $a\in X$ is a $(T_g)$-
transitive point and $U\subset X$ is a non-empty open 
set. 

Indeed, any transitive $G$-system $(X,(T_g))$ that admits a 
continuous function whose orbit separates points (as in condition~(\ref{cond iii}) of Definition \ref{FS_definition}) is a Furstenberg 
system of some bounded function $f:G\to \C$. Given a 
transitive point $a\in X$ and a bounded continuous 
function 
$F\in C(X)$ whose orbit separates points, we can simply 
consider $f:G\to \C$ given by $f(g)=F(T_ga)$, for each 
$g\in G$. If the system is not transitive but admits 
such a function, then passing to an orbit closure 
yields a subsystem that is a Furstenberg system of 
$f: G \to \mathbb{C}$. In that sense, any system $(X,(T_g))$ is a wasteful and generalised Furstenberg system of some bounded
function $f:G\to \C$.

If $f:\N \to \{1,\ldots,r\}$ is a periodic function, 
say $f(n)=n \pmod{r}$, then its Furstenberg system is 
isomorphic to the rotation $(\T,R_{1/r})$.

If $f:G\to \{1,\ldots,r\}$ is a random colouring, or
more precisely, a sequence of independent and uniformly distributed random variables, namely 
$\mathbb{P}(\{f(g)=i\})=1/r$, then almost surely, the 
Furstenberg system of $f$ is the full shift on 
$r$-symbols, $(\{1,\ldots,r\}^G,(R_g))$.

Furstenberg systems of multiplicative 
functions have lead to interesting connections between 
ergodic theory/topological dynamics and number theory. 
For example, the well-known Chowla conjecture can be 
reformulated as saying that the topological Furstenberg 
system of the Liouville function 
$\lambda:\N\to \{-1,1\}$ is a Bernoulli shift $(\{-1,1\}^\N,\sigma)$. For many deep results, fascinating connections with number theory, 
and interesting examples of diverse behaviours of 
sequences, through Furstenberg systems we refer the reader to \cite{Frantz_2, Frantz_1, Frantz_Host, Frantz_Lem_Rue, Gomilko_Lem_Rue, Kouts_Moragues}.

One may also consider dynamical systems $(X,(T_g))$ with 
anti-actions of the acting semigroup 
$G$ by continuous transformations, by requiring that $T_{gh}=T_hT_g$, 
for all $g,h\in G$. Observe that, when $G$ is a group, an anti-action of $G$ is simply an action of 
$G^{-1}=\{g^{-1}:g\in G\}$, defined by $(T'_{g^{-1}})_{g\in G}=(T_g)_{g\in G}$. 
Of course, when $G$ is abelian any $G$ action is also an 
anti-action and vice versa. Replacing actions with anti-
actions in Definition \ref{FS_definition} we may speak 
of the \textit{Furstenberg anti-stystem} of a function 
and we remark that all the results we present for 
actions, can be carried out for anti-actions as well. 
For completeness we include without proof
the following simple observation. 
\begin{proposition}
A semigroup $G$ is abelian if and only if the Furstenberg 
anti-system of any bounded function $f:G\to \C$ is also its Furstenberg system.    
\end{proposition}

\subsection{Ultrafilters and the maximal Furstenberg system}\label{section:Gbeta}
Recall that a \textit{compactification} of a topological space $X$ is a 
compact Hausdorff space $Y$ that contains $X$ as a dense subspace. Discrete spaces are completely regular, and so we can consider their Stone-\v{C}ech compactification, which is essentially the maximal one. More precisely, given a discrete semigroup $G$, we denote by $\beta G$  its Stone-\v{C}ech compactification, which is a compactification with the universal property that any bounded (continuous) function $f:G\to \C$ (indeed, any function $f:G\to K$, where $K$ is compact Hausdorff) extends uniquely to a continuous function $\beta f: \beta G\to \C$. In fact, there exists a continuous function $i_G: G\to \beta G$ such that $f=\beta f \circ i_G$.

In this setting, $\beta G$ can be described as the space of ultrafilters on $G$, endowed with the topology whose basis consists of the sets $\overline{A} = \{p \in \beta G : A \in p\}$ for $A \subset G$, making it compact 
Hausdorff. 

Moreover, $G$ can be embedded densely into $\beta G$ by 
identifying each $g \in G$ with the \emph{principal 
ultrafilter at $g$}, which we also denote by $g$, defined 
via
$g = \{A \subset G : g \in A\}.$
Under this identification, $i_G(g) = g$ for all $g \in G$.

The operation of the semigroup $G$ naturally extends to $\beta G$, which
then becomes a right topological semigroup, $(\beta G,\cdot)$. For completeness, if $p,q\in \beta G$, then 
$$p \cdot q=\{A\subset G: \{g\in G: \{h\in G: gh\in A\}\in q\}\in p\}.$$

Given $p\in \beta G$ fixed we denote by $\lambda_p:\beta G\to \beta G$, the 
map $\sigma_p: q\mapsto q\cdot p$, which is continuous. We will further 
need to consider limits along ultrafilters which are defined as follows. Let $q\in \beta G$ and $f: G\to \C$ be a bounded function in $G$. Then, 
$$q-\lim_{g\in G} f(g)=x$$ if for any open neighborhood $U$ of $x$, it 
holds that $\{g\in G: f(g)\in U \}\in q$. It is well-known that for 
functions taking values on compact Hausdorff spaces, these limits 
always exist and are unique (but not independent of the ultrafilter 
$q$). Moreover, as we explain in the proof of Theorem \ref{maximalFS}, it actually holds that $\beta f(q)=q-\lim_{g\in G} f(g)$, for any $q\in \beta G$. 

Now, let $f:G\to \C$ be a bounded function. We proceed to construct 
the maximal generalised Furstenberg system of $f$. This choice of 
terminology is justified by Theorem \ref{FS are isomorphic intro}, as this system is maximal with 
respect to the partial ordering induced by topological factor maps.

Consider the continuous function 
$\beta f: \beta G\to \C$ described above. Note that, for any $g\in G$, it 
holds that $\sigma_g \beta f(e)=\beta f(g)=f(g)$, where $g$ also denotes 
the principal ultrafilter at $g$ and $e$ denotes the principal ultrafilter 
at the identity element, $e$, of $G$. Since $G$ is densely embedded in $\beta G$ it follows that $\{\sigma_g(e): g\in G\}$ is dense in $\beta G$. Therefore, $(\beta G, (\sigma_g))$ is a dynamical $G$-system with $e$ being a transitive point and $\sigma_g \beta f(e)=f(g)$, for any $g\in G$. By definition, $(\beta G, (\sigma_g))$ is a generalised Furstenberg system of $f$. 

We also explain an alternative way to essentially recover the symbolic Furstenberg 
system of $f$, this time with phase space in $\C^{\beta G}$.
Again, we consider the continuous extension $\beta f: \beta G\to \C$ and 
the dynamical $G$-system $(\beta G_f,(\sigma_g))$, where $\beta G_f$ is the 
orbit closure of $\{\sigma_g (\beta f): g\in G\}$. Note also that $\beta G_f \subset C(\beta G)$, because we endow $C(\beta G)$ with the product 
topology, which is the topology of pointwise convergence and therefore, $$q-\lim_{g\in G} \sigma_g \beta f(p)=p-\lim_{g\in G} q-\lim_{h\in G} \sigma_{gh} \beta f=p\cdot q-\lim_{g\in G}\sigma_g\beta f=\beta f(p \cdot q)=\beta f \circ \sigma_q (p),$$ for each $p,q\in \beta G$ and $\beta f \circ \sigma_q $ is continuous.  
Let $F\in C(\beta G_f)$ 
be the evaluation at $e$, the principal ultrafilter at the identity element 
of $G$, namely $F(G)=G(e)$, for $G\in \beta G_f$. The $(\sigma_g)$-orbit of 
$F$ separates points because $G$ is densely embedded in $\beta G$ and 
hence, any two continuous functions (in $\beta G_f$) that agree on $G$ will 
be identical. Then, notice that $F(\sigma_g (\beta f))=\beta f(g)=f(g)$, for
any $g\in G$. Thus, indeed, $(\beta G_f,(\sigma_g))$ is a Furstenberg 
system of $f$.

\
\section{Uniqueness of Furstenberg Systems} 
\label{sec:3}

\subsection{Uniqueness}\label{sec:3.1}

In this section we will prove one of our main results, Theorem
\ref{FS are isomorphic intro}. Throughout, we will consider semigroups $G$ 
and bounded functions $f:G\to \C$. However, we invite the reader to notice 
that all we will demand from $f(G)$ (for the proofs to be carried out) is
that it is contained 
in a compact Hausdorff space.

We begin by observing that any system that is
isomorphic to a Furstenberg system of some function will 
also be its Furstenberg system. More precisely, if 
$\mathbb{X}=(X,(T_g))$ is a 
(generalised/wasteful) Furstenberg system of $f$ and 
$\mathbb{Y}=(Y,(S_g))$ is topologically 
isomorphic to $\mathbb{X}$, then $\mathbb{Y}$ is also a 
(generalised/wasteful) Furstenberg system of $f$. We leave the easy details of this fact to the interested reader.

We proceed to prove the first part of Theorem 
\ref{FS are isomorphic intro}, namely that Furstenberg 
systems associated to the same function are isomorphic.

\begin{theorem}\label{thm:Uniqueness}
Let $G$ be a semigroup and $f:G\to \C$ be bounded. Then, any two Furstenberg 
systems of $f$
are topologically isomorphic.      
\end{theorem}

\begin{proof}
Let $(X, (T_g)_{g\in G})$ be a Furstenberg system of $f$, with transitive point $a\in X$ and $F\in C(X)$ as in Definition \ref{FS_definition}. Recalling the symbolic construction presented in (\ref{def:X_R}), it is suffices to show that $(X, (T_g)_{g\in G})$ is topologically isomorphic to $(X_{R,f},(R_g)_{g\in G})$. To this end, we claim that the map
$$\phi_R: X\mapsto X_{R,f},\ \phi_R(x)=(F(T_gx))_{g\in G}$$
is a topological isomorphism between $(X, (T_g)_{g\in G})$ and
$(X_{R,f}, (R_g)_{g\in G})$.

We aim first to show that, for any $x\in X$, $\phi_R(x)\in X_{R,f}$. Consider an arbitrary cylinder set $C=C(U_1,\ldots,U_k\ ;\ g_1,\ldots,g_k)$ containing $\phi_R(x)$. This means that $T_{g_i}x\in V_i:=F^{-1}(U_i)$, for each $i=1,\ldots,k$ and by continuity we have that $V_1,\ldots,V_k\subset X$ are also open sets. Now, since $G$ acts on $X$ by continuous maps, $F\in C(X)$ and $a\in X$ is transitive, we can find $h\in G$ such that $T_{g_ih}a=T_{g_i}(T_ha)\in V_i$, for all $i=1,\ldots,k$. Recall that $F(T_{g_ih}a)=f(g_ih)$ and so we obtain that $f(g_ih)\in U_i$, for all $i=1,\ldots,k$. In other words, there exists $h\in G$ such that $R_h(f)\in C$. Since $C$ is an arbitrary cylinder set containing $\phi_R(x)$ and $R_h(f)\in X_{R,f}$, this shows that $\phi_R(X)\subset X_{R,f}$.

The proof of the continuity of the map $\phi_R$ is 
implicitly contained in the previous argument. For, in order to prove continuity, it suffices to show that 
for any cylinder $C=C(U_1,\ldots,U_k\ ;\ g_1,\ldots,g_k)$ that intersects $\phi_R(X)$ non-trivially, $\phi^{-1}C$ contains an open set.
To this end, we can use the continuity of $T_{g_1},\ldots,T_{g_k}$ and $F$ to find an open set $V\subset X$ such $\phi_R(V) \subset C$. 

To verify that $\phi_R$ commutes the dynamics of 
$(T_g)$ and $(R_g)$, observe that, for every $x\in X$ and $h\in G$,
\begin{equation}\label{eq:2}
\phi_R(T_h x)=(F(T_g(T_hx)))=(F(T_{gh}x))=R_h(F(T_gx))=R_h(\phi_R(x)).    
\end{equation}

The injectivity of $\phi_R$ follows immediately from the equivalent conditions~(\ref{cond iii}) and ~\ref{cond ii} in Definition~\ref{FS_definition}. To prove that $\phi_R$ is surjective, we first note that $\phi_R(X)\subset X_{R,f}$ is 
compact, since $\phi_R$ is continuous, and thus closed, 
as $X_{R,f}$ is Hausdorff. Moreover,
$$\phi_R(a)=(F(T_g a))_g = (f(g))_g \in \phi_R(X)$$ 
is a transitive point in $X_{R,f}$. From \eqref{eq:2} it follows that $\phi_R(X)$ contains the orbit of $\phi_R(a)$ and therefore its orbit closure, which by transitivity, equals $X_{R,f}$.  

Finally, since continuous bijections from compact to Hausdorff spaces are homeomorphisms, we can conclude that $\phi_R$ is a homeomorphism.
\end{proof}

Having established that Furstenberg systems are isomorphic, we may speak of \textit{the Furstenberg system} associated with a given function $f \colon G \to \C$, choosing a representative among these isomorphic systems. 

An almost obvious consequence of this is that 
the Furstenberg system of $f:G\to \C$ is a subsystem of any wasteful Furstenberg system of it. Indeed, if $(X,(T_g)_g)$ is a wasteful Furstenberg system, where $F(T_g a)=f(g)$ for some continuous function $F$ and some $a \in X$, we can recover the Furstenberg system of $f$, by restricting to the subsystem $X'=\overline{\{T_ga: g\in G\}}$. On the other hand, we can trivially add isolated points to the Furstenberg system of a function $f$, and recover a wasteful system, making it obviously non-unique.

The situation becomes more interesting when considering 
generalised Furstenberg systems, namely when the 
function $F$ witnessing the behaviour of $f$ is not 
required to separate distinct
orbits. We show that the Furstenberg system can be 
recovered as a factor of any generalised system. 
Moreover, there exists a maximal Furstenberg system 
which contains all generalised Furstenberg systems as 
its factors. This is the the content of the second part of Theorem \ref{FS are isomorphic intro}.

\begin{theorem} \label{maximalFS}
Let $G$ be a semigroup and $f:G\to \C$ be a bounded function. Let $(X,(T_g))$ be a generalised 
Furstenberg system of $f$, and let $(\beta G, (\sigma_g))$ be the system defined in 
Section \ref{section:Gbeta}. Then, the Furstenberg system of $f$ is a factor 
of $(X,(T_g))$, which in turn is a factor of $(\beta G, (\sigma_g))$. In this sense, we can call $(\beta G, (\sigma_g))$ the \emph{maximal Furstenberg system}.
\end{theorem}
\begin{proof}
As $(X,(T_g))$ is a generalised Furstenberg system, there is a continuous function $F\in C(X)$ and a transitive point $a\in X$, such that $F(T_ga)=f(g)$, for all $g\in G$. 
To prove the first part of the theorem, by Theorem~\ref{thm:Uniqueness}, it suffices to show that the symbolic system $(X_{R,f}, (R_g))$ is a factor of $(X, (T_g))$, since any two isomorphic systems have the same extensions. To this end we consider the (a posteriori factor) map 
$$\phi_R: x \mapsto (F(T_gx))_{g\in G}$$
from $(X, (T_g)_{g\in G})$ to $(X_{R,f}, (R_g)_{g\in G})$. To see this is indeed a factor map, we revisit the proof of Theorem~\ref{thm:Uniqueness}, noting that the map is no longer injective in general, since the function $F$ does not necessarily distinguish orbits. Nevertheless, surjectivity, continuity, and commutation of the dynamics of $(T_g)$ and $(R_g)$ follow in exactly the same manner.

Now, we show that any generalised Furstenberg system 
$(X, (T_g))$ is a factor of $(\beta G, (\sigma_g))$. 

Let $a \in X$ and $F \in C(X)$ be as before.  
Since $a$ is $(T_g)$-transitive, $G$ can be densely 
embedded in $X$, and because $X$ is compact, it 
forms a compactification of $G$. Then, the 
continuous function $T:G \to X$ defined via $g \mapsto T_ga$ 
can be lifted continuously to $\beta T: \beta G \to X$ 
by the universal 
property of the Stone-\v{C}ech compactification. Note 
that one function 
(which is unique up to homeomorphisms) with 
this property is 
$\beta T_{\ell}(q)=q-\lim_{g\in G} T_ga$, 
for $q\in \beta G$. Indeed, let $U\subset X$ be any open neighborhood of $\beta T_{\ell}(q)$. As $X$ is compact Hausdorff, and thus, in particular, a regular space, we can find an open neighborhood $V\subset U$ of $\beta T_{\ell}(q)$ such that $\overline{V}\subset U$. By the definition of limits along ultrafilters, $A=\{g\in G: T_ga \in V\}\in q$. Hence, $\overline{A} \subset \beta T_{\ell}^{-1}(\overline{V}) \subset \beta T_{\ell}^{-1}(U)$, showing that $\beta T_{\ell}$ is continuous. Also, clearly, for a principal ultrafilter $g\in G$, we have $\beta T_{\ell}(g)=T_ga=T(g)$.

So, we assume $\beta T=\beta T_{\ell}$. Recall the
definition of the ultrafilter Furstenberg system from 
Section \ref{section:Gbeta}. We will show that the map 
$\beta T: (\beta G, (\sigma_g))\to (X,(T_g))$ defined above is a factor map. First, $\beta T$ is continuous as explained already. Next, observe that $\beta T$ is surjective because 
$$\beta T(G)=\{T(g):g\in G\}=\{T_ga:g\in G\},$$ 
is dense in $X$, $\beta T$ is continuous, $\beta G$ is compact and $X$ is 
Hausdorff. Finally, let $q\in \beta G$ and $h\in G$ be any. Then,
$$\beta T(\sigma_hq)=\beta T(qh)=T_h\beta T(q),$$
where we have used that for any $h\in G$ and the principal ultrafilter $h\in \beta G$, 
$$T_h \left(q-\lim_{g\in G} T_ga\right)=q-\lim_{g\in G} T_{hg}a=hq-\lim_{g\in G} T_ga.$$ 
\end{proof}
We emphasise that generalised Furstenberg systems associated with the same function
need not be isomorphic. In particular, there exist 
functions with a generalised Furstenberg system that is 
genuinely distinct from (i.e. not isomorphic to) their 
Furstenberg system and from their maximal Furstenberg 
system. 
Indeed, if $f$ is periodic and $G = \mathbb{Z}$, the associated Furstenberg system
is finite. A generalised Furstenberg system can then be obtained by taking the
direct product of the (unique) Furstenberg system with an irrational rotation,
yielding an infinite system. In contrast, the maximal system is non-metrizable,
and therefore these notions cannot coincide. The following diagram illustrates how the maximal system controls, through factor maps,
all Furstenberg systems, while the (generalised) Furstenberg system depends on the
specific function.
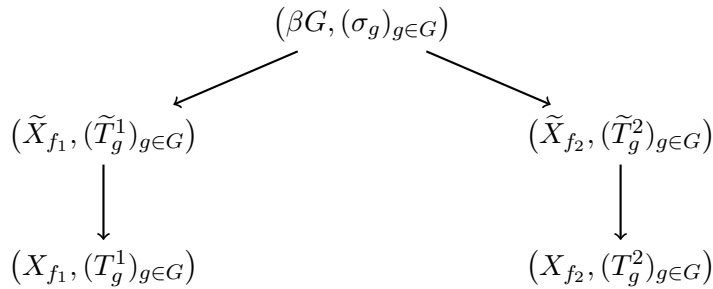
\begin{center}
\begin{tikzpicture}[
  node distance=1cm,
  every node/.style={draw=none},
  arrow/.style={->, thick}
]

\node (top) {$\bigl(\beta G, (\sigma_g)_{g \in G}\bigr)$};

\node (leftmid) [below left=of top] 
  {$\bigl(\widetilde{X}_{f_1}, (\widetilde{T}^1_g)_{g \in G}\bigr)$};

\node (rightmid) [below right=of top] 
  {$\bigl(\widetilde{X}_{f_2}, (\widetilde{T}^2_g)_{g \in G}\bigr)$};

\node (leftbot) [below=of leftmid] 
  {$\bigl(X_{f_1}, (T^1_g)_{g \in G}\bigr)$};

\node (rightbot) [below=of rightmid] 
  {$\bigl(X_{f_2}, (T^2_g)_{g \in G}\bigr)$};

\draw[arrow] (top) -- (leftmid);
\draw[arrow] (top) -- (rightmid);
\draw[arrow] (leftmid) -- (leftbot);
\draw[arrow] (rightmid) -- (rightbot);

\end{tikzpicture}
    
\end{center}
\subsection{Joint Furstenberg systems of multiple functions}\label{joint fs section}

In analogy with Definition \ref{FS_definition}, in 
order to find a dynamical 
model for the statistics of multiple functions simultanesouly, we can consider the 
\textit{joint Furstenberg system} of a finite set of bounded functions over a semigroup.

\begin{definition}\label{jointFS}
Let $G$ be a semigroup and let $f_1,\ldots,f_{\ell}: G\to \C$ be bounded. A 
dynamical $G$-system $(X,(T_g)_{g\in G})$ is called a joint Furstenberg 
system of $f_1,\ldots,f_{\ell}$ if 
\begin{enumerate}[(i)]
    \item \label{cond joint_i} there exist functions $F_1,\ldots,F_{\ell} \in C(X)$ and a transitive point $a\in X$ such that $F_i(T_ga)=f_i(g)$, for all $g\in G$, $i\in \{1,\ldots,\ell\}$ and
    \item \label{cond joint_iii} for any $x\neq y$ in $X$, there exists $g\in G$ and $i\in \{1,\ldots,\ell\}$ such that $F_i(T_gx)\neq F_i(T_gy)$.
\end{enumerate}
Note that condition (\ref{cond joint_iii}) above can be replaced by 
\begin{enumerate}[(iii)] 
    \item \label{cond joint_ii} $C(X)$ is the smallest $(T_g)$-invariant $C^{*}$-algebra generated by $(F_1,\ldots,F_{\ell})$.
\end{enumerate}
We also consider a \textit{generalised Furstenberg system} of 
$f_1,\ldots,f_{\ell}:G\to \C$, by removing condition~(\ref{cond joint_iii}) above. 
\end{definition}

Note that, in Definition \ref{jointFS}, we may consider 
a multi-valued function $F: X\to \C^{\ell}$ defined via 
$F(x)=(F_1(x),\ldots,F_{\ell}(x))$, and rewrite condition~(\ref{cond joint_iii}) as the 
property that the $G$-orbit of $F$ separates points.

As the prototypical example, we consider the symbolic right shift joint
Furstenberg system of a finite family of functions. If $G$ is a semigroup, and $f_1,\ldots,f_{\ell}:G\to \C$ are bounded, then we consider $f\in (\C^{\ell})^{G}$ defined via $f(g)=(f_1(g),\ldots,f_{\ell}(g))$ and in analogy with the symbolic system constructed in Section \ref{sec:2.1}, we consider $(X_{R,f},(R_g))$ to be the orbit closure of $f$ under $(R_g)$, only this time, the phase space is a subspace of $(\C^{\ell})^{G}$. With the continuous function $F: X_{R,f}\to C^{\ell}$ defined by $F(x)=(x_1(0),\ldots,x_{\ell}(0))$, we recover a joint Furstenberg system of $f_1,\ldots,f_{\ell}$. In an analogous manner, the maximal joint Furstenberg system of $f_1,\ldots,f_{\ell}$ can be defined using the Stone-\v{C}ech compactification of $G$, $(\beta G, (\sigma_g))$.

As we drew attention to in the beginning of this section, nothing special 
about the structure of the image of single functions $f:G\to \C$ was used 
in the proofs of the Section \ref{sec:3.1}, other than the fact that
$\C$ is
Hausdorff and that $\overline{f(G)}$ is compact, since $f$ was assumed bounded. The same assumptions hold when considering $f=(f_1,\ldots,f_{\ell})$. Namely, $f(G)\subset \C^{\ell}$ is a pre-compact set of a Hausdorff space, and thus all the results proved for a single function extend in the case of finite tuples of functions.  

\begin{theorem}\label{joint FS results}
Let $G$ be a semigroup and $f_1,\ldots,f_{\ell}:G\to \C$ be bounded. If $(X,(T_g))$ is a 
(generalised) joint Furstenberg system of $f_1,\ldots,f_{\ell}$ 
and $(Y,(S_g))$ is topologically 
isomorphic to $(X,(T_g))$, then $(Y,(S_g))$ 
is also a (generalised) joint 
Furstenberg system of $f_1,\ldots,f_{\ell}$. At the same time, any two 
joint Furstenberg systems of $f_1,\ldots,f_{\ell}$
are topologically isomorphic, and they are factors of any generalised joint 
Furstenberg system of $f_1,\ldots,f_{\ell}$, which in turn is a factor 
$(\beta G, (\sigma_g))$.   
\end{theorem}

\begin{remark*}
As the above discussion suggests, all these results are 
special cases of,
but share the exact same proof with, the obvious 
generalisation for functions $f:G\to K$, where $K$ is a 
compact Hausdorff space. 
\end{remark*}

\subsection{Isomorphic sets}\label{isomorphic subsection} 

We now restrict attention to the group setting. 
For any $A \subset G$, where $G$ is a group, we can identify $A$ with 
$a := (\mathbbm{1}_A(g))_{g \in G}$.  
Under this identification, we may consider the Furstenberg system of $A$. Our main goal for this section, is to understand when two subsets $A, B \subset G$ are \emph{isomorphic sets}, meaning that their Furstenberg systems are topologically isomorphic.

A first simple observation is that a set $A \subset G$ is
always isomorphic to its complement, $G \setminus A$, via 
an isomorphism of their symbolic Furstenberg systems 
swapping $0$'s with $1$'s. More generally, it is not too
difficult to prove that, if $B$ is a right or 
left shift of $A$ (or its complement), then $A$ is 
isomorphic to $B$. More precisely, assume there exists $h\in G$ and a permutation $\sigma:\{0,1\}\to \{0,1\}$, such that $\mathbbm{1}_{B}(g)=\sigma(h_x \mathbbm{1}_{A}(g))$, for any $g\in G$, where $h_x\in \{R_h,L_h\}$. Then, $A$ and $B$ are isomorphic.

Although this sufficient condition is a natural one, it 
does not 
capture all the possibilities for when two sets are 
isomorphic. For
example, let $(X,T)=(X,(T_n)_{n\in\Z})$ be an infinite 
and minimal subsystem of $(\{0,1\}^{\Z},T)$, where $T$ is 
the right shift. Such a system is generated by a 
non-periodic, uniformly recurrent element of 
$\{0,1\}^{\Z}$. Then, if $a\in X$ is any element and 
$b\in X\setminus \{T^na:n\in \Z\}=\overline{\{T^na:n\in \Z\}}\setminus \{T^na:n\in \Z\}$, we see that $(X,T)$ 
is a Furstenberg system of both $a$ and $b$. In particular, if we let $A, B \subset \mathbb{Z}$ be the sets corresponding to the sequences $a$ and $b$, then $A$ and $B$ are isomorphic. However, there exists no shift 
$m \in \mathbb{Z}$ and no permutation 
$\sigma: \{0,1\} \to \{0,1\}$ such that $a = \sigma \circ T^m b$. In other words, neither $A$ nor $G \setminus A$ 
is a shift of $B$.  Notice, however, that in this example 
we have 
$b \in \overline{\{T^n a : n \in \mathbb{Z}\}}$ 
and 
$a \in \overline{\{T^n b : n \in \mathbb{Z}\}}$.

Given $x\in \{0,1\}^G$, a permutation $\sigma:\{0,1\}\to \{0,1\}$ and a set $Y\subset \{0,1\}^G$, we write $x\in_{\sigma} Y$, if there exists some $y\in Y$ such that $x(g)=\sigma(y(g))$, for all $g\in G$. 
The previous observation naturally leads to the 
following, more general, sufficient condition. 
\begin{proposition} \label{limit case of isomorphic systems sufficient}
Let $G$ be a group, $A,B\subset G$ and let $a=\mathbbm{1}_A$, $b=\mathbbm{1}_B$ in $\{0,1\}^G$. If there exists a
permutation $\sigma:\{0,1\}\to \{0,1\}$ such that 
$b\in_{\sigma} \overline{\{R_ga:g\in G\}}$ and 
$a\in_{\sigma} \overline{\{R_gb:g\in G\}}$, then $A$ 
and $B$ are isomorphic.     
\end{proposition}

\begin{remark*}
The result states that if each point is in the orbit 
closure of the other one (up to changing $0$'s with $1$'s, i.e. taking complements), then they must be isomorphic. A more combinatorial reformulation of this result was stated in Proposition \ref{isomorphic sets intro}.
\end{remark*}

\begin{proof} By Theorem \ref{thm:Uniqueness}, it suffices to show that $(X_{R,a}, (R_g))$ is also a Furstenberg system of $b$. Thus, we aim to find a continuous function $F$, whose orbit separates points, and a transitive point $x\in X_{R,a}$, such that $F(R_gx)=b(g)$, for all $g\in G$. By assumption, $b\in_{\sigma} X_{R,a}$, so there exists $x\in \overline{\{R_ga:g\in G\}}$ satisfying $(\sigma(x(g)))_{g\in G}=b$. Moreover, $a\in_{\sigma} \overline{\{R_gb:g\in G\}}$, and thus $x$ is an $(R_g)$-transitive point in $X_{R,a}$. To see this note (the obvious fact) that $(x(g))_{g\in G}=(\sigma^2(x(g)))_{g\in G}=(\sigma(b(g)))_{g\in G}$.

We now simply consider $F$ to be the projection onto the 
coordinate at the identity, $F_e:y\mapsto y(e)$, 
composed with $\sigma$, which is continuous and its orbit separates points. It then
follows that 
$$F(R_gx)=\sigma\circ F_e(R_gx)=\sigma(x(g))=b(g),$$ for all $g\in G$.
\end{proof}

Unfortunately -- at least from an aesthetic point of view -- this still leaves some cases unaccounted for. 
We first present an example demonstrating this, and then 
proceed to state and prove general necessary and 
sufficient conditions.

\begin{proposition}\label{example of isomorphic points}
In $(\{0,1\}^{\mathbb{Z}}, T)$, where $T((x(n))) = (x(n+1))$, let $a,b$ be defined by
\[
a(n) = 1 \iff n \in (2\mathbb{N} + 1) \cup \{0\}, 
\qquad
b(n) = 1 \iff n \in 2\mathbb{N}.
\]
Then $a$ and $b$ are isomorphic, but 
\[
a \notin \overline{\{T^n b : n \in \mathbb{Z}\}}
\qquad \text{and} \qquad 
b \notin \overline{\{T^n a : n \in \mathbb{Z}\}}.
\]
\end{proposition}

\begin{proof}
The last claim is clear because no two consecutive $1$'s appear in $b$, whereas $a(0)=a(1)=1$. To prove that the 
the two points, and hence the sets corresponding to 
their indicators, have isomorphic Furstenberg systems, we 
show that $(X_b,T)$, 
where $X_b=\overline{\{T^nb:n\in \Z\}},$ is a Furstenberg 
system of $a$. Observe that this suffices because $(X_b,T)$ is the (symbolic) Furstenberg system of $b$.

We consider $F:x\mapsto x(0)+x(1)+x(2) \pmod{2}$ on $X_b$. It is 
clear that $F$ is continuous. Moreover, 
$$F(T^nb)=b(n)+b(n+1)+b(n+2) \pmod{2} = 1 \iff n\in \{0,1,3,5,\ldots\},$$
so that $F(T^nb)=a(n)$, for every $n\in \Z$. Since $b$ is transitive 
by definition, we are left with showing that the orbit of $F$ 
separates points in $X_b$. To this end, observe that by 
the definition of $b$, $X_b$ will 
consist of the $0$ 
constant sequence, the two bi-infinite sequences of 
consecutive 
blocks of $01$ and all points of $x_k,y_k$ defined via 
$$x_k(n)=1 \iff n\geq k\ \text{and}\ n\in 2\Z,$$ and
$$y_k(n)=1 \iff n\geq k\ \text{and}\ n\in (2\Z+1),$$
where $k\in \Z$. It is then easy to check that the orbit 
of $F$ separates these points.  
\end{proof}

To account for this type of isomorphic points we
introduce some useful 
notation. Given a function $F: \{0,1\}^G\to \C$, and $a,b\in \{0,1\}^G$, we write that $a\in_{F} \overline{\{R_gb:g\in G\}}$ if there exists $x\in \overline{\{R_gb:g\in G\}}$ such that $F(R_gx)=a(g)$, for every $g\in G$, which we will denote by $a=_{F}x$. Notice that for $F$ being the projection onto the coordinate at the identity, the above simply means that $a\in \overline{\{R_gb:g\in G\}}$. We further say that $F: \{0,1\}^G\to \C$ \textit{depends on finitely many parameters} if there exist $g_1,\ldots,g_k$ such that for any $x,y\in \{0,1\}^G$ with $x(g_i)=y(g_i)$, for all $i=1,\ldots,k$, it holds that $F(x)=F(y)$. If $F:\overline{\{R_gb:g\in G\}} \to \C$ is continuous, depends on finitely many parameters and its $G$-orbit separates points, then we call $F$ a \textit{generalised projection on $\overline{\{R_gb:g\in G\}}$}. In particular, observe that the function $F:x\mapsto x(0)+x(1)+x(2) \pmod{2}$ used in the proof of Proposition \ref{example of isomorphic points} is a generalised projection on $X_b=\overline{\{T^nb:n\in \Z\}}$, with $b$ as in the same proposition.

We can give necessary and sufficient conditions for 
when two subsets of a group are isomorphic via the 
following generalisation of Proposition 
\ref{limit case of isomorphic systems sufficient}. Before stating the result, we recall the following version of the classical Curtis--Hedlund--Lyndon theorem (see Theorem 6.2.9 in \cite{LindMarcus}). For completeness, we include a proof adapted to our setting. 

\begin{theorem}[Curtis-Hedlund-Lyndon]\label{CHL}
Let $G$ be a group, and let $f,f':G\to \mathcal{A}$ be finite-valued functions. Suppose $\phi$ is a factor map between the Furstenberg systems $(X_{R,f},(R_g)_g)$ and $(X_{R,f'},(R_g)_g)$. Then, there exists $g_1,\dots,g_k$ and a local code $\psi:\mathcal{A}^k\to \mathcal{A}$ such that
$$\phi(x)_h=\psi(x_{g_1h},\dots,x_{g_kh})$$
for all $x\in X_{R,f}$ and $h\in G$.
\end{theorem}
\begin{proof}
Let $a \in \mathcal{A}$, and suppose the cylinder set 
\[
C(a) = \{y \in X_{R,f'} : y_{e_G} = a\}
\]
is non-empty. By continuity of $\phi$, the preimage $\phi^{-1}(C(a))$ 
is open, and hence can be written as a (possibly uncountable) union of 
cylinder sets $(C'(a,i))_{i \in I_a}$ in $X_{R,f}$. 

Repeating this for each $a \in \mathcal{A}$ yields an open cover 
\[
(C'(a,i))_{a \in \mathcal{A},\, i \in I_a}
\]
of $X_{R,f}$. Since $X_{R,f}$ is compact, we can extract a finite 
subcover $(C_i)_{i \in I}$. Each $C_i$ depends only on finitely many 
coordinates and is contained in exactly one set of the form 
$\phi^{-1}(C(a))$. 

It follows that there exist $g_1, \dots, g_k \in G$ such that whenever 
$x_{g_i} = y_{g_i}$ for all $i = 1, \dots, k$, we have 
\[
\phi(x)_{e_G} = \phi(y)_{e_G}.
\]
Consequently, there exists a function $\psi : \mathcal{A}^k \to \mathcal{A}$ 
such that
\[
\phi(x)_{e_G} = \psi(x_{g_1}, \dots, x_{g_k}).
\]
Finally, since $R_h \circ \phi = \phi \circ R_h$ for all $h \in G$, we obtain that
\[
\phi(x)_h 
= (R_h \phi(x))_{e_G} 
= (\phi(R_h x))_{e_G} 
= \psi(x_{g_1 h}, \dots, x_{g_k h}).
\]

\end{proof} 

\begin{theorem}\label{characterising isomorphic sets}
Let $G$ be a group and let $a,b\in \{0,1\}^G$. Then, $a,b$ are 
isomorphic if and only if there exists a generalised projection
$F:\overline{\{R_gb:g\in G\}} \to \C$ and a transitive point 
$x\in \overline{\{R_gb:g\in G\}}$, such that 
$a\in_F \overline{\{R_gb:g\in G\}}$ and $a=_F x$.    
\end{theorem}

\begin{proof}
In order to prove sufficiency, it is enough to show that 
$(X_{R,b},(R_g))$ is a Furstenberg system of $a$. This follows 
directly from the assumptions and the definition of a Furstenberg 
system via the quadruple $(X_{R,b},(R_g), F,x)$.

For necessity, suppose $\phi: (X_{R,a},(R_g)) \to (X_{R,b},(R_g))$ is 
an isomorphism between the Furstenberg systems of $a$ and $b$. Then $\phi^{-1}$ is 
a factor map that commutes with $(R_g)$, and thus, by Theorem \ref{CHL}, it follows that $\phi^{-1}$ is given 
via a local code (in the sense of Theorem \ref{CHL}). Now, let $F_e$ denote the projection onto the coordinate at the identity, and let $F=F_e\circ \phi^{-1}$. As $\phi$ is an isomorphism and $a$ is transitive in $X_{R,a}$, it follows that $x=\phi(a)$ is also transitive in $X_{R,b}$. Moreover, $a\in_F \overline{\{R_gb:g\in G\}}$ and $a=_F x$. Indeed, for any $g\in G$, we have that 
$$F(R_gx)=F_e(\phi^{-1} (R_g\phi(a)))=F_e(\phi^{-1} (\phi(R_ga)))=F_e(R_ga)=a(g),$$
showing that $a=_{F}x$. Finally, observe that $F$ is a generalised 
projection on $X_{R,b}$. Clearly, it is continuous and depends on finitely many parameters as the composition of a projection with a local code. Finally, we check that the orbit of $F$ separates points in $X_{R,b}$. If $y,w\in X_{R,b}$ are distinct, then $\phi^{-1}(y),\phi^{-1}(w)\in X_{R,a}$ are also distinct, since $\phi$ is a bijection. But then the $(R_g)$-orbit of $F_e$ separates the distinct points $\phi^{-1}(y),\phi^{-1}(w)$, and since $\phi^{-1}$ commutes with $R_g$, it follows that the $(R_g)$-orbit of $F=F_e\circ \phi^{-1}$ separates $y$ and $w$.
\end{proof}

\begin{remark*}
The assumption that $x$ is transitive is necessary, because, for 
example, $a\in \{0,1\}^{\Z}$ defined via $a(n)=1\iff n\geq 0$, is not isomorphic to the $0$ constant sequence, but $0\in \overline{\{T^na:n\in \Z\}}$. Also, we need $F$ to satisfy all the properties of a generalised projection. Indeed, we consider $a,b\in \{0,1\}^{\Z}$ defined via $a(n)=1$, if and only 
if $n\in \N$, and $b(n)=1$ if and only if $n\in 2\N$, and 
let $F:x\mapsto x(0)+x(1)\pmod{2}$. Notice that $F$ is a continuous function that depends on finitely many parameters. Then, $a\in_F \overline{\{T^nb:n\in \Z\}}$, 
and in fact, $a=_{F} b$. However, the points $a$ and $b$ 
are not 
isomorphic and this doesn't contradict Theorem \ref{characterising isomorphic sets}, because the $\Z$-orbit of $F$ does not separate points in $\overline{\{T^nb:n\in \Z\}}$. Indeed, the two limit points of $b$, $x_0,x_1$ defined via $x_0(n)=1 \iff n\in 2\Z$ and $x_1(n)=1 \iff n\in (2\Z+1)$, are such that $F(T^nx_0)=F(T^nx_1)=1$, for every $n\in \Z$.
\end{remark*}

\subsection{Distinguishing the left shift from Furstenberg Systems}\label{sec:left_shifts}
Let $G$ be a group and $f \colon G \to \mathbb{C}$ a bounded function. Consider the left shift action $L_h(x) = (x(h^{-1}g))_g$ on the compact space $\overline{f(G)}^G$. As discussed in Section~2.2, one may ask whether the system
\begin{equation} \widetilde{X}_{L,f}=\overline{\Big\{L_h(f(g))_{g\in G}:\ h\in G\Big\}}=\overline{\Big\{(f(h^{-1}g))_{g\in G}:\ h\in G\Big\}}\end{equation}
defines a Furstenberg system for $f$. If we take the transitive point $(f(g))_g$ and let $F$ be the projection onto the $e_G$-coordinate, then $F(L_h f) = f(h^{-1})$. This shows that, unless $f(h^{-1}) = f(h)$ for all $h \in G$, this choice of transitive point and observable $F$ does not capture the behaviour of $f$. However, this does not rule out the existence of a different transitive point in the same space and some continuous function that do realise the dynamics of $f$, making $\Tilde{X}_{L,f}$ still represent a Furstenberg system for $f$.

It turns out that, even for finite groups, the answer to 
this question depends heavily on the structure of the 
group. However, when $G$ is abelian, the system does indeed 
define a Furstenberg system for any choice of $f$. This 
fact, however, does not extend to the infinite 
setting: even for relatively simple infinite abelian 
groups, the statement fails.

A group $G$ is said to be Dedekind if every subgroup of $G$ 
is normal. This condition allows us to characterise the 
property of $(\widetilde{X}_{L,f},(L_g))$ being a 
Furstenberg system. Obviously, every abelian group is 
Dedekind.

\begin{theorem}
    Let $G$ be a finite group.  Then, $(\widetilde{X}_{L,f}, (L_g)_{g\in G})$  defines a Furstenberg system for every function $f:G\to \mathbb{C}$ if and only if $G$ is Dedekind.
\end{theorem}
\begin{proof} By Theorem \ref{FS are isomorphic intro}, we know that the Furstenberg system of a function is unique up-to topological isomorphism. Hence, it is enough to check whether we can find a topological isomorphism between $(\widetilde{X}_{L,f}, (L_g)_{g\in G})$ and $(X_{R,f}, (R_g)_{g\in G})$.

    Firstly, suppose $G$ is Dedekind and let $f:G\to \mathbb{N}$ be a function, and consider its respective system $(\tilde{X}_{L,f}, (L_g)_{g\in G})$.  Let us assume the existence of $b_1\in G\setminus\{e\}$ such that $f(x)=f(b_1x)$ for all $x\in G$. If there exists $b_2\notin \langle b_1\rangle$ such that $f(x)=f(b_2x)$ for all $x\in G$, it follows that, for any $c\in \langle b_1,b_2\rangle $, $f(x)=f(cx)$ for all $x\in G$. More generally, we can consider a maximal subgroup $H$ such that $f(x)=f(cx)$ for all $x\in G$ if and only if $c\in H$. In case $b_1$ as above does not exist, we simply take $H=\{e\}$. In addition, it is easy to see that $f$ is constant on equivalence classes of $G/H$ We define the map
\begin{equation}
    \begin{aligned}
\varphi : (X_{R,f}, (R_g)_{g\in G}) &\to (\widetilde{X}_{L,f}, (L_g)_{g\in G}) \\
\Big(f(gh)\Big)_{g\in G} &\mapsto \Big(f(h^{-1}g)\Big)_{g\in G} 
\end{aligned}\label{finite_case_1}
\end{equation}

First, we emphasise that the arguments for well-definedness and injectivity
are analogous. For simplicity, we only present the proof of
injectivity. Let us consider $h,\Tilde{h}\in G$ such that $$\varphi \Big((f(gh))_{g\in G}\Big)=\varphi \Big((f(g\Tilde{h}))_{g\in G}\Big) \Leftrightarrow f(h^{-1}g)=f(\Tilde{h}^{-1}g) \text{ for all $g\in G$}.$$
Equivalently, $f(x)=f(\Tilde{h}^{-1}hx)$ for all $x\in G$.
Then, $\Tilde{h}^{-1}h\in H$. Since $G$ is Dedekind, $H$ is a normal subgroup, and therefore $[gh]_{H}=[g\Tilde{h}]_H$. This implies that $f(gh)=f(g\Tilde{h})$ for all $g\in G$, concluding that $\varphi$ is injective. Since $G$ is finite and discrete, $\varphi$ is obviously continuous. Also, it is easy to check that $\varphi \circ R_h=L_h\circ \varphi$ for every $h\in G$ and $\varphi$ is surjective. Therefore, $\varphi$ a topological isomorphism.

Now, we assume that $G$ is not Dedekind. Let $H$ be a 
non-normal subgroup of $G$. Then, we can find $a\in G$ such that $aH\neq Ha$. In particular, there is $b\in H$ such that $ba\notin aH$, and thus $ba\notin a\langle b\rangle $. We consider $G/\langle b \rangle=\{g_1,\dots,g_{\ell}\}$, where each element represents a different coset. We define $$f:G \mapsto \{1,\dots,\ell\},\ \text{via}\ f(x)=i, \ \text{whenever}\ x\in \langle b\rangle g_i.$$ We claim that $(\widetilde{X}_{L,f},(L_g)_g)$ is not a Furstenberg system of $f$.

     By contradiction, suppose there is a topological isomorphism $\varphi$ from $(X_{R,f},(R_g)_g)$ to $(\widetilde{X}_{L,f},(L_g)_g)$. Since $X_{R,f}$ and $\widetilde{X}_{L,f}$ are  finite, there exists $T:G\to G$ such that 
    $$\varphi\Big( (f(gh))_{g\in G}\Big)=(f(T(h)g))_{g\in G}.$$
    Since $\varphi \circ R_h=L_h\circ\varphi$, it follows that $T(h)=T(e)h^{-1}:=ch^{-1}.$
   By definition, $f(x)=f(bx)$ for all $x\in G$. Therefore, if $h:=c^{-1}b^{-1}c$, $$f(ch^{-1}g)=f(bcg)=f(cg)=f(ce^{-1}g) \text{ for all $g\in G.$}$$ Hence $\varphi((f(gh)))_g)=\varphi((f(g))_g).$
   Since $\varphi$ is injective, $$f(gh)=f(gc^{-1}b^{-1}c)=f(g)$$ for every $g\in G$. Then, $$\langle b \rangle g c^{-1}b^{-1}c=\langle b\rangle g.$$ 
   Taking $g=a^{-1}c$, $\langle b\rangle a^{-1}b^{-1}=\langle b\rangle a^{-1}$. However, this implies that $a^{-1}\in \langle b \rangle a^{-1}b^{-1}$, which in turn implies that $ba \in a\langle b\rangle$, leading to a contradiction. 
    \end{proof}
It is worth noting from the proof that, if for every $b \neq e$ there exists $x_b \in G$ such that $f(x_b) \neq f(bx_b)$, then no additional assumption on the group is needed, and the systems $X_{R,f}$ and $\widetilde{X}_{L,f}$ are isomorphic.

When passing to the infinite setting, the situation becomes substantially more delicate. In this case, continuity is no longer automatic, and the presence of limit points introduces difficulties in establishing this property, especially in groups where $g$ and its inverse behave differently. A trivial case is when the group $G$ is \emph{Boolean}, i.e., $g^2 = e$ for every $g \in G$. In Boolean groups, the actions $R_h(x) = (x(gh))_g$ and $L_h(x) = (x(h^{-1}g))_g$ coincide, and therefore $(\widetilde{X}_{L,f}, (L_g)_g)$ is the same system as $(X_{R,f}, (R_g)_g)$.

It is plausible that these examples essentially exhaust all such cases. 
The following technical condition show that $(\widetilde{X}_{L,f}, (L_g)_g)$ 
typically fails to be a Furstenberg system beyond this rigid setting.
    \begin{lemma}\label{tech}
           Let $G$ be an infinite group, and suppose there exists a finite-valued function $f:G\to \mathbb{C}$ such that $f(y)=0$ if and only if $y=e_G$, and with the following property: For every finite $S\subset G$, there exists $c\in G$ and $x\in G$ such that 
           \begin{equation}
               f(x)\neq f(cx), \text{ and } f(a x^{-1} b)=f(a(cx)^{-1}b) \text{ for all $a,b\in S$} .\label{eq_tech}
           \end{equation}
          Then, $(\widetilde{X}_{L,f},(L_g)_g)$ is not a a Furstenberg system of $f$.
    \end{lemma}
\begin{proof}
    Suppose that $\widetilde{X}_{L,f}$ is a Furstenberg system. By Theorem~\ref{thm:Uniqueness}, there exists a topological isomorphism
$\varphi$ between $(X_{R,f}, (R_g)_g)$ and $(\widetilde{X}_{L,f}, (L_g)_g)$.
Following the same idea as in the proof of Theorem~\ref{CHL}, we modify the argument and use the relation
$L_h \circ \varphi = \varphi \circ R_h$ to obtain elements $g_1, \ldots, g_k \in G$ such that, for every $y \in X_{R,f}$ and $h \in G$,  $$\varphi(y)|_{h}=\Big(L_{h^{-1}}\varphi(y)\Big)_0=\varphi\Big(R_{h^{-1}}y\Big)_0=\varphi\Big((y_{gh^{-1}})_g\Big)_0=\psi(y_{g_1h^{-1}},\cdots,y_{g_kh^{-1}}).$$
  Note that in $X_{R,f}$, every element of the form $(f(gh))_g$ is transitive. Suppose that $y\in X_{R,f}$ is transitive but not of that form. The transitivity implies the existence of $h\in G$ such that $y(h)=f(e)$. Since $(f(g))_g$ is transitive, for every finite set $S\cup \{h\}\subseteq G$ we can find $u\in G$ such that $f(gu)=y(g)$ for all $g\in S\cup \{h\}$. In particular, $f(hu)=y(h)=f(e)$, so necessarily $u=h^{-1}$. Since $S$ is arbitrary, it must be that $x=R_u (f(g)_g).$ Since $\varphi$ is bijective, in particular, there is $y^{*}\in X_{R,f}$ such that 
    $$\varphi(y^{*})|_h=f(h)=\psi(y^{*}_{g_1h^{-1}},\cdots,y^{*}_{g_kh^{-1}})$$ for all $h\in G$. Since $\varphi(y^{*})$ is transitive and and $\varphi$ is a topological conjugacy, then $y^{*}$ must be transitive as well. Then, there exists $h^{*}$ such that
    $$y^{*}=(f(gh^{*}))_g$$
    Therefore, for all $h\in G$,
    \begin{equation}
        f(h)=\psi\Big( f(g_1h^{-1}h^{*}),\cdots,f(g_kh^{-1}h^{*})\Big).\label{eq_tech2}
    \end{equation}
Let $S=\{g_1,\cdots,g_k,h^{*}\}.$ By the assumption, we can find $c,x\in G$ such that (\ref{eq_tech}) holds. However, (\ref{eq_tech2}) and the fact that $f(a x^{-1} b)=f(a(cx)^{-1}b) \text{ for all $a,b\in S$}$ implies $f(x)=f(cx),$ leading to a contradiction.
\end{proof}
    
For instance, the previous conditions applies in any group with a non-trivial homomorphism with $\mathbb{R}$. However, it also applies in much more general cases, including when moving sightly from of boolean groups.
\begin{corollary}
    Consider the boolean group $$B := \bigoplus_{i=1}^\infty \mathbb{Z}/2\mathbb{Z}
= \{(x_i)_i : x_i \in \mathbb{Z}/2\mathbb{Z},\ x_i \neq 0 \text{ for only finitely many } i\},$$
and let $G:=\Z/3\Z \times  B$. Then, there exists a finite-valued function $f:G\to \mathbb{C}$ such that $(\widetilde{X}_{L,f},(L_g)_g)$ is not the Furstenberg system of $f$.
\end{corollary}

\begin{proof}
We define $f : G \to \mathbb{C}$ by setting $f(e_G) = 0$, and for 
$g = (g_1, g_2) \in \mathbb{Z}/3\mathbb{Z} \times B$ with $g \neq e_G$,
\[
f(g) =
\begin{cases}
1 & \text{if } g_1 = 1 \text{ and } g_2 = 1^{2k} \text{ for some } k \geq 1, \\
1 & \text{if } g_1 = 2 \text{ and } g_2 = 1^{2k+1} \text{ for some } k \geq 0, \\
-1 & \text{otherwise.}
\end{cases}
\]
In view of Lemma~\ref{tech}, the proof follows by checking that the function 
$f$ satisfies ~\eqref{eq_tech}. 

\end{proof}

\section{Minimal Systems, Minimal Sets, and Topological Recurrence}
\label{sec:5}
\subsection{Minimality in systems and sets}\label{sec:minimal}
Let $G$ be a semigroup. A set $S\subset G$ is \textit{left-syndetic} if 
there exist finitely many elements $g_1,\ldots,g_k\in G$ such that $(g_1^{-1}S) \cup \cdots \cup (g_k^{-1}S)=G$, where $g_i^{-1}S$ is well-defined even if $g_i$ is not invertible as the set $\{g\in G: g_ig \in S\}$. 
One can define right-syndetic sets in the obvious way, and in commutative
semigroups there is no distinction between left- and right-syndeticity. For general semigroups these are no longer equivalent notions and since we restrict our attention to systems defined by semigroup actions (although the same results hold for anti-actions as well), for our purposes, it turns out we will only need to consider left-syndetic sets. 

Now, let $(X,(T_g))$ be a topological dynamical $G$-system, or \textit{system} for 
short. A point $x\in X$ is called \textit{uniformly recurrent} if for any 
open neighborhood $V\subset X$ of $x$, the set $\{g\in G: T_gx \in V\}$ is 
left-syndetic. 

The following lemma is proven in \cite[Lemma 1.14]{Fur_book} for
abelian semigroups, but the proof is the same in full generality.

\begin{lemma}\label{minimal systems by coverings}
A system $(X,(T_g))$ is minimal if and only if for any non-empty open set $V\subset X$, there exist $g_1,\ldots,g_k$ such that $\bigcup_{i=1}^k T_{g_i}^{-1}V=X$. 
\end{lemma}

The next result also appears in \cite{Fur_book} for 
abelian semigroups. Here, 
lack of commutativity forces us to consider left-syndetic 
sets and to showcase this fact we include the proof of 
the first part of this proposition.      

\begin{proposition}\label{minimal systems and uniform recurrence}
Let $(X,(T_g))$ be a system. If $\mathbb{X}$ is minimal, then every point $x\in X$ is uniformly recurrent. In a partial converse of that, if $x\in X$ is a uniformly recurrent point, then $(Y,(T_g))$, where $Y=\overline{\{T_gx: g\in G\}}$, is a minimal system.
\end{proposition}

\begin{proof}
Assume $\mathbb{X}$ is minimal. Then, we claim that for every $x\in X$, and
every non-empty open set $U\subset X$, the set $V(x,U)=\{g\in G: T_gx\in U\}$ 
is left-syndetic. Indeed, by Lemma \ref{minimal systems by coverings} we find $g_1,\ldots,g_k\in G$ such that $\bigcup_{i=1}^k T_{g_i}^{-1}U=X$. But then, for any $g\in G$ there exists $i\in \{1,\ldots,k\}$, so that $T_gx\in T_{g_i}^{-1}U$. This implies that $T_{g_ig}x\in U$, so that $g\in g_i^{-1}V(x,U)$. This shows that $V(x,U)$ is indeed left-syndetic.
\end{proof}

The goal of this section is to obtain a characterization of \textit{minimal sets}, that is, those with minimal Furstenberg systems. We first introduce the relevant terminology necessary. Given a semigroup $G$, a set $H\subset G$ is called \textit{dynamically syndetic} if it is of the form $H=\{g\in G: T_ga\in U\}$, for some non-empty open set $U\subset X$, where $(X,(T_g))$ is a minimal $G$-system. Dynamically syndetic sets are always left-syndetic (this is implicit in the proof of Lemma \ref{minimal systems and uniform recurrence}), but not every syndetic set is dynamically syndetic (see, for example, \cite[(1.1)]{Glass_Le}). For more on dynamically syndetic sets -- on the integer setting -- we refer the reader to \cite{Glass_Le}. 

It turns out that minimal sets correspond to a special subclass of dynamically syndetic sets which we call \textit{strongly dynamically syndetic}. Let us first justify the need to consider a subclass of dynamically syndetic sets by showing that it is strictly larger than the class of minimal sets. 

\begin{proposition}
Not every dynamically syndetic set is minimal.    
\end{proposition}

\begin{proof}
Indeed, consider an irrational rotation on 
the torus, $(\T,R_{\alpha})$, for some $\alpha \notin \Q$, $\alpha<1/2$. Then, 
$A=\{n\in \Z: n\alpha \notin [\alpha,2\alpha]\}$ is dynamically syndetic, 
since the rotation is minimal, but we claim it is not a minimal 
set. Considering its Furstenberg system, and letting $a=\mathbbm{1}_A \in \{0,1\}^{\Z}$, it suffices to show that $a$ is not uniformly recurrent. In view of Lemma \ref{uniformly recurrent symbolic points} below (in the integer setting this appears in \cite[Proposition 1.22]{Fur_book}) we simply have to find a word that appears in $a$, but does not occur syndetically. Note that if $a(n)=0$, then $n\alpha \in [\alpha,2\alpha]$. Hence, the only way that $a(n)=a(n+1)=0$ is if $n\alpha=\alpha$ and $(n+1)\alpha=2\alpha$, which, since $\alpha\notin \Q$, happens exactly once, when $n=1$. In particular, the word $00$ occurs in $a$, but not syndetically. 
\end{proof}

Before defining strongly dynamically syndetic sets, we 
first remark that it is not too difficult to see that 
minimal sets are the dynamically syndetic sets which 
correspond to visit times in clopen sets. Indeed, one
of the implications directly follows by considering the 
symbolic Furstenberg system of a minimal set, and 
observing that cylinder subsets of $\{0,1\}^G$ are 
clopen. 
For the converse, if $H=\{g\in G: T_ga\in U\}$, for 
some non-empty clopen set $U\subset X$, in a minimal system $(X,(T_g))$, then, $(X,(T_g))$ is a 
generalised Furstenberg system of $H$, through the continuous 
function $F=\mathbbm{1}_{U}$. By Theorem \ref{FS are isomorphic intro} (or Theorem \ref{maximalFS}), we can find a factor map from $(X,(T_g))$ to a Furstenberg system of $H$. As factors of minimal systems are minimal, we conclude that $H$ is a minimal subset of $G$.

We can actually give a better, that is to say, seemingly 
more general, characterization of minimal sets. 

\begin{definition}\label{strongly dynamically syndetic sets}
A set $H\subset G$ is called \textit{strongly dynamically syndetic} if it can be written as $\{g\in G: T_ga\in U\}$, where $(X,(T_g))$ is a minimal $G$-system, $U\subset X$ is some non-empty open set, and, in addition, the set $\{g\in G: T_ga\in \overline{U}\setminus U\}$ is empty.     
\end{definition}

Note, in particular, that if $U$ is a clopen set, then $\overline{U}\setminus U=\emptyset$, hence $\{g\in G: T_ga\in \overline{U}\setminus U\}$ is trivially empty. 
Moreover, it is not a priori obvious that strongly dynamically syndetic sets are the same as dynamically syndetic sets on clopen target sets. 

\begin{theorem} \label{characterizing minimal sets}
Let $G$ be a semigroup. Then, a subset of $G$ has a minimal 
Furstenberg system if and only if it is strongly dynamically syndetic.
\end{theorem}

\begin{remark}\label{minimal set theorem discussion}
We stress that a minimal set $H$ will, in general, have 
many different representations as a set of hitting 
times of some orbit in distinct dynamical systems. The 
result states that if we can find one representation as 
a strongly dynamically syndetic set, then the set is 
minimal. The utility of such a result lies in the fact 
that it allows one to verify that a set has a minimal 
Furstenberg system by an oftentimes simpler procedure 
than checking the Furstenberg system directly. Consider
for example the set 
$H=\{n\in \Z: n\alpha \in (1/3,2/3)\}$, for some 
$\alpha\notin \Q$. It is not too difficult to check 
directly, that is, using the symbolic construction, 
that the Furstenberg system of $H$ is 
indeed minimal, but it certainly is simpler to verify 
that $H$ is strongly dynamically syndetic. 
\end{remark}

Before giving the proof of Theorem \ref{characterizing minimal sets} we introduce some necessary 
terminology and an intermediate lemma.

Let $G$ be a semigroup and $a=(a(g))_{g\in G}\in \Lambda^G$, where $\Lambda$ 
is some finite alphabet. We define a word in
$a$ to be any tuple of the form $a_1\cdots a_k;g_1,\dots,g_k$, where  
$a_1\cdots a_k\in \Lambda^k$ and $g_1,\dots,g_k\in G$ are such that 
$a(g_1h)=a_1,\dots,a(g_kh)=a_k$, for some $h\in G$. In this case we say that \textit{the word $a_1\cdots a_k;g_1,\dots,g_k$ occurs at $h$}. Recall our standing assumptions that semigroups under consideration are enriched with an identity element. For the results we are about to prove the identity could be replaced by any fixed element of the semigroup. The reason for this -- as the tentative reader may observe in what follows -- is that if $e\in G$ is any element and $S\subset G$ is a left-syndetic set, then $eS$ is also left-syndetic. In any case, we denote the identity element by $e$, and speak of an \textit{initial word} in $a$ if the word occurs at $e$.

\begin{lemma}\label{uniformly recurrent symbolic points}
Let $G$ be a semigroup and $\Lambda$ be a finite alphabet. A point $a\in \Lambda^{G}$ is uniformly recurrent for the right shift if and only if, every initial word in $a$, occurs along a left-syndetic set.   
\end{lemma}

\begin{proof}
In $\Z$ this appears in \cite[Proposition 1.22]{Fur_book}. Assume that $a$ is 
uniformly recurrent and let $a_1\cdots a_k;g_1,\dots, g_k$ be a word that 
occurs at $e$ in $a$. 
Then, there exists a cylinder set 
$$C=C(a_1,\ldots,a_k;g_1,\ldots,g_k)=\{x=(x(g))\in \Lambda^G: x(g_i)=a_i,\ 
\text{for each}\ i=1,\ldots,k\}$$
so that $a\in C$ and notice that $C$ is a clopen set. By the definition of 
uniform recurrence, the set $V(a,C)=\{g\in G: T_ga\in C\}$ is left-syndetic, 
where $T_g$ denotes the right shift by $g\in G$. But for 
any $h\in V(a,C)$ 
we see that the word $a_1\cdots a_k; g_1,\dots, g_k$ 
occurs at $h$, and hence this word occurs along a 
left-syndetic set. Indeed, $h\in V(a,C)$ means that 
$T_ha\in C$ and thus $T_ha(g_i)=a(g_ih)=a_i$, for each $i=1,\ldots,k$.

Conversely, assume that every initial word in $a$ occurs along a left-syndetic set and let $C\subset \Lambda^{G}$ be a basic open set that 
contains $a$, namely a cylinder $C=C(a_1,\ldots,a_k;g_1,\ldots,g_k)$ as in the previous part of the proof. In order to prove that $a$ is uniformly recurrent it suffices to 
show that $V(a,C)$, defined as before, is left-syndetic. Unravelling the 
definitions, $a\in C$ means that the word 
$a_1\cdots a_k;g_1,\dots, g_k$ occurs at $e$, so it is an initial word. By assumption, it follows that $V(a,C)$ is left-syndetic, because 
$h\in V(a,C)$ if and only if $a_1\cdots a_k;g_1,\dots, g_k$ occurs at $h$.
\end{proof}

We can now prove Theorem \ref{characterizing minimal sets}.

\begin{proof}[Proof of Theorem \ref{characterizing minimal sets}]
Note that by Theorem \ref{FS are isomorphic intro}, a set has a minimal Furstenberg system if and only if its symbolic Furstenberg system is minimal. Now, assume $H\subset G$ is a minimal set identified with its indicator function $f:G\to \{0,1\}$ and let $(X_{R,f},(T_g))$ be the right shift symbolic Furstenberg system of $f$. Then, by assumption, $(X_{R,f},(T_g))$ is a minimal system and by the definition of $(X_{R,f},(T_g))$ we can represent $H$ as $\{g\in G: T_gf\in U\}$, where $U=\{x\in X_{R,f}: x(e_G)=1\}$ is a non-empty clopen set. 

For the converse direction, assume $H=\{g\in G: T_ga\in U\}$, for 
some non-empty open set $U\subset X$, where $(X,(T_g))$ is a 
minimal $G$-system and further assume that $\{g\in G: T_ga\in \overline{U}\setminus U\}=\emptyset$. We consider the orbit closure $X_{R,H}$ of $\mathbbm{1}_H \in \{0,1\}^G$ and claim that $(X_{R,H},(R_g))$ is a minimal system. Since this is a symbolic Furstenberg system of $H$ this concludes the proof. To prove the claim, by Proposition \ref{minimal systems and uniform recurrence}, it suffices to show that $\mathbbm{1}_H$ is uniformly recurrent. To this end, we will apply Lemma \ref{uniformly recurrent symbolic points}.

Let $a_1\cdots a_k;g_1,\dots,g_k$ be a word that occurs at $e$ in 
$\mathbbm{1}_H$, where $a_1,\dots,a_k\in \{0,1\}$ and $g_1,\ldots,g_k\in G$. By definition, $T_{g_j}a\in U$ if and only if $\mathbbm{1}_H(g_j)=1$ if and only if $a_j=1$. By our assumption, if $T_{g_j}a \notin U$, then $T_{g_j}a \in X\setminus\overline{U}$. So, $T_{g_j}a$ are interior points of either $U$ or its complement. Let $U_j=U$ if $a_j=1$ and $U_j=X\setminus\overline{U}$ if $a_j=0$.
By continuity of $T_{g_1},\ldots,T_{g_k}$, there is an open neihborhood of $a$, $V=V_{g_1,\ldots,g_k}$, such that $T_ga\in V$ implies that $T_{g_j}T_ga=T_{g_jg}a\in U_j$, for all $j=1,\ldots,k$. Since $(X,(T_g))$ is minimal we see that $\{g\in G: T_ga\in V\}$ is 
left-syndetic, which in turn implies that the word $a_1\cdots a_k;g_1,\ldots,g_k$ occurs in $\mathbbm{1}_H$ 
left-syndetically.   
\end{proof}

Observe that nothing too specific about $\{0,1\}$-valued 
functions  was used 
in the arguments above that prohibits analogous results for finite-valued 
functions, which correspond to finite colourings of $G$, just as $\{0,1\}$-valued functions correspond to subsets of $G$. In fact, Lemma 
\ref{minimal systems and uniform recurrence} was proven in the general form 
for finite alphabets. We say that a \textit{colouring} $G=C_1\cup \cdots \cup C_r$ -- which is simply a partition of $G$ into finitely many cells -- is 
\textit{minimal}, if the Furstenberg system of $f:G\to \{1,\ldots,k\}$ defined via 
$f(g)=i$ if and only if $i\in C_i$, is minimal. 

\begin{theorem}\label{minimal colourings}
Let $G$ be a semigroup. A colouring $G=C_1\cup \cdots \cup C_r$ is 
minimal if and only if there is a minimal system $(X,(T_g))$, a point 
$a\in X$ and a partition $X=U_1\cup \cdots \cup U_k$ into sets with non-empty 
interior, such that $C_i=\{g\in G: T_ga\in U_i^{\circ}\}$, for each 
$i=1,\ldots,k$.
\end{theorem}

This almost immediately implies that for a mimimal colouring, every colour 
must be a minimal set itself. 

\begin{corollary}\label{minimal colourings corollary}
If a colouring $G=C_1\cup \cdots \cup C_r$ is 
minimal, then each colour is a minimal set. 
\end{corollary}

\begin{proof}
Assume $G=C_1\cup \cdots \cup C_r$ is a minimal colouring. By Theorem 
\ref{characterizing minimal sets} we have to show that each colour is 
strongly dynamically syndetic, namely that, in the notation of Theorem \ref{minimal colourings}, $\{g\in G: T_ga\in \overline{U_i}\setminus U_i\}=\emptyset$, for each $i\in \{1,\ldots,k\}$. But if $g\in G$ is such that $T_ga\notin U_1$, say, then $T_ga \in U_i^{\circ}$ for some $i\in \{2,\ldots,k\}$. Since $\overline{U_i} \cap U_j^{\circ}=\emptyset$, for $i\neq j$, the result follows. 
\end{proof}

This observation raises the natural question 
of whether the converse holds. At least for two colourings, this is 
trivially true. Indeed, if a set $C\subset G$, with $C\neq G$, is minimal, 
then so is its complement and so is the $2$-colouring 
$G=C \cup (G\setminus C)$ because they all share the same Furstenberg 
system (see also Section \ref{isomorphic subsection}). 
However, this is not true in general. 

\begin{proposition}\label{minimal colours non minimal colouring}
There exist non-minimal colourings $G=C_1\cup \cdots \cup C_r$ with all colours being minimal.    
\end{proposition}

\begin{proof}
We once again consider an irrational rotation 
$(\T,R_{\alpha})$, some $\alpha \notin \Q$ and say 
$1/7<\alpha<2/7$. Then, we define the 
sets $C_1=\{n\in \Z: n\alpha\in [0,1/7)\}$, $C_2=\{n\in \Z: n\alpha\in (1/7,2\alpha]\}$, $C_3=\{n\in \Z: n\alpha\in (2\alpha, 6/7\}$ and $C_4=\{n\in \Z:n\alpha\in (6/7,1)\}$. It is clear that $\Z=C_1\cup C_2 \cup C_3 \cup C_4$ is indeed a colouring of $\Z$, since $\alpha \notin \Q$. Moreover, the Furstenberg system of this colouring is not minimal. To see this, note that if $\chi \in \{1,2,3,4\}^{\Z}$ is the colouring, the word $122$ occurs exactly once at $0$. For, $\chi(0)=1$ and $\chi(1)=\chi(2)=2$ and, if $n\neq 0$ is such that $\chi(n)=1$, then $n\alpha \in (0,1/7)$ and thus $(n+2)\alpha \in (2\alpha,2\alpha+1/7)$, hence $\chi(n+2)=3$.

Unfortunately, none of the sets $C_1,C_2,C_3,C_4$ is 
represented as a strongly syndetic set and thus we can 
not directly apply Theorem \ref{characterizing minimal sets} to show they are indeed minimal. Nevertheless, this follows from Proposition \ref{characterizing minimal sets stronger} below. 
\end{proof}

As we explained in Remark 
\ref{minimal set theorem discussion}, the utility of Theorem \ref{characterizing minimal sets} lies in that it allows us to conclude that a set is minimal in a simple way, if a particular representation of a set is available. But as we saw in the proof of Proposition \ref{minimal colours non minimal colouring}, the definition of strongly dynamically syndetic sets may be somewhat restrictive; indeed, it does not immediately cover sets like $H=\{n\in \Z: R^n0\in U\}$, where $R$ denotes the rotation by $\alpha \notin \Q$ and $U=(0,1/3)$, which are minimal. Notice that in $H$ there exists exactly one \textit{bad point} in the orbit of $0$, namely a point $R^n0$ such that $R^n0\in \overline{U}\setminus U$ and that, moreover, any open neighbourhood of this bad point intersect $U^{c}$ with non-empty interior. One may refer to a set $H$ with such a representation as \textit{almost strongly dynamically syndetic} and it turns out that the existence of such a bad point can be ramified so that the proof of Theorem \ref{characterizing minimal sets} essentially carries through in this case as well. 

\begin{proposition}\label{characterizing minimal sets stronger}
Let $G$ be a semigroup. Suppose $H=\{g\in G: T_ga\in U\}$, where $(X,(T_g))$ is a minimal $G$-system, and $U\subset X$ is a non-empty open set. Moreover, assume there exists a single bad point $T_ba\in \overline{U}\setminus U$, for some $b\in G$, but any open neighborhood of $T_ba$ intersects $X\setminus U$ with non-empty interior. Then, $H$ is minimal.    
\end{proposition}

\begin{proof}
Suppose $H=\{g\in G: T_ga\in U\}$ as in the statement. 

A first reduction is that without loss of generality we 
may assume $b=e$, the identity element of $G$. Indeed, 
$Hb^{-1}=\{g\in G: gb\in H\}=\{g\in G: T_g(T_{b}a)\in U\}$, and $\{g\in G: T_g(T_ba)\in \overline{U}\setminus U\}=\{e\}$. It follows by the results in Section \ref{isomorphic subsection} that $H$ is minimal if and only if $Hb^{-1}$ is. 

Under these assumptions we proceed as in the proof of 
Theorem \ref{characterizing minimal sets}. Indeed, all 
we need to check is that any word in $\mathbbm{1}_H$ 
that occurs at $e$, which includes the bad coordinate, namely $a_1\cdots a_k;g_1,\dots,g_k$, where some $g_i=e$, occurs left-syndetically. For all the other 
words that occur in $\mathbbm{1}_H$, we can use the 
argument of the proof of Theorem \ref{characterizing minimal sets}.

By assumption, we can find an open neighborhood $V$ of $a$, where $a\in \overline{U}\setminus U$, (note that $V\setminus U$ has non-empty interior) and then, a non-empty open set $U(g_1,\ldots,g_k)\subset V\setminus U$, such that if $T_ha\in U(g_1,\ldots,g_k)$ we also have that $T_{g_ih}a \in U \iff T_{g_i}a\in U$. The fact that $a\in V\setminus U$ gives $T_{h}a\in X\setminus U$. Thus, for every $h\in U(g_1,\ldots,g_k)$ the word $a_1\cdots a_k;g_1,\dots,g_k$ occurs at $h$, and since $U(g_1,\ldots,g_k)$ is open and the system is minimal, the word occurs left-syndetically.
\end{proof}

\subsection{Sets of topological recurrence}
\label{top_recurrence}
Our next goal is to connect Furstenberg systems with 
\textit{sets of topological recurrence} which have been extensively studied in the literature (see, for example, \cite{Donoso_Hernandez_Maass, Forest, GHSW, GKLMRR, Glas_Kouts_Rich, Griezmer1, Griezmer2, Host_Kra_Mass, Huang_Shao_Ye, Katznelson, Kriz}). A particularly important role for this purpose is taken by minimal sets studied in Section \ref{sec:minimal}. 
Recall that a set $S\subset G$ is called minimal if its Furstenberg system 
is a minimal system and that this happens precisely when $S$ is strongly dynamically syndetic, according to Definition \ref{strongly dynamically syndetic sets}.  

Given a semigroup $G$ and a set $R\subset G$, we let $R^{*}=R\setminus \{e_G\}$. We also let $g^{-1}S=\{h\in G: gh\in S\}$, for $g\in G$ and $S\subset G$.

\begin{definition}\label{sets of recurrence}
A set $R\subset G$ is a set of topological recurrence if for any 
minimal action $(T_g)_{g\in G}$ on a compact space $X$ and any 
non-empty open set $U\subset X$, there is $g\in R^{*}$, such that 
$U\cap T_g^{-1}U\neq \emptyset$.     
\end{definition}

\begin{remark*}
Definition \ref{sets of recurrence} sometimes has the 
assumption that the space $X$ is also metrizable. The two 
definitions turn 
out to be actually equivalent, and we use the above in 
order to 
be able to consider symbolic systems with phase space the 
form 
$X=K^G$, where $K\subset\C$ is compact and $G$ an 
arbitrary 
semigroup, so that $X$ may not be metrizable.
\end{remark*}

Sets of topological recurrence have combinatorial counterparts.

\begin{definition}
A set $R\subset G$ is a set of chromatic recurrence, or, chromatically intersective, if for any 
finite colouring $G=C_1\cup \cdots \cup C_r$, there is a colour 
class $C_i$, $i\in \{1,\ldots,r\}$, and some $g\in R^{*}$, such 
that $C_i \cap g^{-1}C_i \neq \emptyset$.    
\end{definition}

It is well-known (see, for example, \cite{McCutcheon} for 
a proof in the integer case) that a 
set $R\subset G$ is a set of topological recurrence if 
and only if it is a set of chromatic recurrence if and 
only if for any
left-syndetic set $S\subset G$, there is $g\in R^{*}$ 
such that 
$S\cap g^{-1}S \neq \emptyset$. 

Observe that a finite colouring 
of $G$ is simply a finite-valued function defined on $G$. Given
the equivalence of sets of recurrence with sets of chromatic 
recurrence, it should come as no surprise that we can actually 
show equivalence with the seemingly weaker notion of \textit{{minimal chromatic 
intersectivity}}, namely chromatic intersectivity for minimal colourings (these are finite-valued functions on $G$ with minimal Furstenberg systems).

Furthermore, minimal sets are left-syndetic, as follows by Theorem 
\ref{characterizing minimal sets} and the proof of Proposition \ref{minimal systems and uniform recurrence}. The converse is not true, namely, there are 
left-syndetic sets which are not minimal. To see this simply take any set 
that is both syndetic and thick in the integers\footnote{A subset of $\Z$ is thick if it contains arbitrarily large intervals, or, equivalently, if it intersects every syndetic set.} with non-empty 
complement, and note that its complement cannot be syndetic, hence it cannot 
be minimal, and so the original set cannot be minimal either (since, 
as explained in Section \ref{isomorphic subsection}, a set and its complement have 
the same topological Furstenberg system).  
In fact, as was mentioned in Subsection \ref{sec:minimal}, it's not even true that all syndetic sets are dynamically syndetic. 

Combining all of the above facts it should perhaps be expected that the 
following holds.  

\begin{proposition}
Let $G$ be a semigroup and let $R\subset G$. The following are equivalent:
\begin{enumerate}[(i)]
    \item $R$ is a set of topological recurrence. 
    \item $R$ is a set of chromatic recurrence.
    \item For any left-syndetic set $S\subset G$, there is $g\in R^{*}$ such that $S\cap g^{-1}S \neq \emptyset$.
    \item For any \textit{minimal} colouring $G=C_1\cup \cdots \cup C_r$, there is a colour class $C_i$, $i\in \{1,\ldots,r\}$, and some $g\in R^{*}$, such that $C_i \cap g^{-1}C_i \neq \emptyset$.
    \item For any dynamically syndetic $H\subset G$,  there is some $g\in R^{*}$ such that $H\cap g^{-1}H \neq \emptyset$.
    \item For any minimal set $H\subset G$, there is some $g\in R^{*}$ such that $H\cap g^{-1}H \neq \emptyset$.
\end{enumerate}
\end{proposition}

\begin{proof}
The equivalence between $(i),(ii)$ and $(iii)$ is well-known. Since minimal colourings are colourings, $(ii)\implies (iv)$. 

Next, we show that $(i) \implies (v)$. Let $H\subset G$ be dynamically 
syndetic. Then, $H=\{g\in G: T_ga\in U\}$, where $(X,(T_g))$ is a minimal 
$G$-system, $U\subset X$ is some non-empty open set. By the definition of
topological recurrence, we find $g\in R^{*}$ such that $U\cap T_g^{-1}U \neq \emptyset$. By minimality, there is $h\in G$ such that $T_ha\in U\cap T_g^{-1}U$. Unravelling the definitions, this shows that $H\cap g^{-1}H\neq \emptyset$. 

Note that the same proof also shows $(i)\implies (vi)$, since minimal sets 
are strongly dynamically syndetic (Theorem \ref{characterizing minimal sets}). One 
can of course deduce $(vi)$ from $(iii)$ independently of this argument, since minimal sets are, in particular, dynamically syndetic. 

Assuming $(iv)$ we now prove $(ii)$. Let $G=C_1\cup \cdots \cup C_k$ be a colouring of $G$ and let $f:G\to \{1,\ldots,k\}$ encode this colouring. We let $U_i$ denote the cylinder sets that correspond to the colours $C_i$. If $(X_{R,f},(R_g))$ is the symbolic Furstenberg system of $f$, then we consider a minimal subsystem of it, $(Y,(R_g))$, which is in particular, a minimal colouring $G=D_1 \cup \cdots \cup D_m$, some $m\leq k$, with cylinder sets of the form $V_i=Y\cap U_i$ corresponding to the colours $D_i$, and this minimal colouring is encoded by some $y\in \{1,\ldots,m\}^G$. Then, by assumption, there is $i\in \{1,\ldots,m\}$ and $g\in R^{*}$ such that $D_i \cap g^{-1}D_i \neq \emptyset$. We can thus find $h\in G$ such that $R_hy \in V_i \cap R_{g}^{-1}V_i$. As $y\in Y$ and $Y$ is a subsystem of the orbit closure of $f$, and since $R_h$ is continuous, we find $u\in G$ such that $R_h(R_uf)\in V_i \cap R_{g}^{-1}V_i$. Unravelling the definitions, this implies that $hu \in C_i\cap g^{-1}C_i$, hence $R$ is a set of chromatic recurrence. 

Finally, a similar proof shows that $(vi)\implies (ii)$. Let $f:G\to \{1,\ldots,k\}$ encode a colouring of $G$, and $Y$ be a minimal subsystem of $(X_{R,f},(R_g))$, using the same notation as before. Then, there is a cylinder $U_i$ such that $V_i=Y\cap U_i\neq \emptyset$. But then, $V_i$ is clopen in $Y$ as $U_i$ is clopen in $X_{R,f}$ and thus $H=\{h\in G: R_hy \in V_i\}$ is strongly dynamically syndetic, hence minimal by Theorem \ref{characterizing minimal sets}. By assumption, there is $g\in R^{*}$ such that $H\cap g^{-1}H\neq \emptyset$, and so there is $h\in G$ such that $R_hy\in V_i \cap g^{-1}V_i$. The rest of the proof of this implication is similar to the previous one. 
\end{proof}

We want to finish by recalling an intriguing conjecture 
of Bergelson 
from \cite{Bergelson} related to sets of topological 
recurrence. The measurable counterpart of this notion is
that of a \textit{set of measurable recurrence}. Given a 
semigroup $G$, a set $R\subset G$ is called a set of 
measurable recurrence if for any measure preserving 
$G$-system, $(X,\mathcal{X},\mu,(T_g))$, and any set $A\in \mathcal{X}$, with $\mu(A)>0$, there exists $g\in R^{*}$ such that $\mu(A\cap T_g^{-1}A)>0$. 

An interesting construction due to K\v{r}\'{\i}\v{z} \cite{Kriz}, describes a set of topological recurrence in $\Z$ that is
not a set of measurable recurrence. It is not too 
difficult to see that sets of measurable recurrence in 
$\Z$ are always sets of topological recurrence; for 
instance by considering a fully supported invariant 
measure on a given minimal $\Z$-system, $(X,T)$. In fact, 
it can be shown that when a group $G$ is \textit{amenable}, sets of measurable recurrence are also sets of topological recurrence. 

Amenable groups are well-known for having a very wide 
range of equivalent definitions. We will not give one here, 
but we refer to \cite{Pier} for a rather comprehensive treatment. 
Bergelson's aforementioned conjecture {\cite[Conjecture, p.55]{Bergelson}} 
proposes yet another definition of amenability via notions of recurrence in measure preserving and dynamical systems.

\begin{conjecture*}{\cite[Conjecture, p.55]{Bergelson}}
A group $G$ is amenable if and only if any set of measurable recurrence $R\subset G$ is a set of topological recurrence.    
\end{conjecture*}

\bibliographystyle{abbrv}

\end{document}